\documentclass[1p]{elsarticle}

\usepackage{amsfonts}
\usepackage{amssymb}
\usepackage{amsthm}
\usepackage{caption}
\usepackage{epstopdf}
\usepackage{graphicx}
\usepackage{hyperref}
\usepackage{mathtools}
\usepackage{pgfplots}
\usepgfplotslibrary{fillbetween}
\usepackage{subcaption}
\usepackage{tikz}
\usepgfplotslibrary{fillbetween}
\usepackage{algorithmic}
\ifpdf
  \DeclareGraphicsExtensions{.eps,.pdf,.png,.jpg}
\else
  \DeclareGraphicsExtensions{.eps}
\fi

\definecolor{mycolor1}{HTML}{ca0020}
\definecolor{mycolor2}{HTML}{f4a582}
\definecolor{mycolor3}{HTML}{58d68d}
\definecolor{mycolor4}{HTML}{92c5de}
\definecolor{mycolor5}{HTML}{0571b0}
\definecolor{mycolor6}{HTML}{a6611a}
\definecolor{mycolor7}{HTML}{dfc27d}

\newenvironment{customlegend}[1][]{%
  \begingroup
  \csname pgfplots@init@cleared@structures\endcsname
    \pgfplotsset{#1}%
  }{%
    \csname pgfplots@createlegend\endcsname
    \endgroup
  }%
  \def\addlegendimage{\csname pgfplots@addlegendimage\endcsname}

\usepackage{enumitem}
\setlist[enumerate]{leftmargin=.5in}
\setlist[itemize]{leftmargin=.5in}

\newtheorem{proposition}{Proposition}[section]
\newtheorem{corollary}{Corollary}[proposition]
\newdefinition{remark}{Remark}

\usepackage{amsopn}

\ifpdf
\hypersetup{
  pdftitle={Nonlinear PageRank problem for local graph partitioning},
  pdfauthor={C. Kodsi and D. Pasadakis}
}
\fi

\begin{document}

\begin{frontmatter}

\title{Nonlinear PageRank Problem for Local Graph Partitioning\tnoteref{t1}}
\tnotetext[t1]{D.P. was financially supported by the joint DFG-470857344 and SNSF-204817 project, and by the Huawei Zurich Research Center. Part of this work was conducted with the help of the Danish Data Science Academy (DDSA) Visit Grant 2023-1855.}

\author[1]{Costy Kodsi\corref{cor1}}
\ead{costy.kodsi02@alumni.imperial.ac.uk}

\author[2]{Dimosthenis Pasadakis}

\cortext[cor1]{Corresponding author}

\affiliation[1]{organization={Department of Mathematical Sciences, Aalborg University},
addressline={Thomas Manns Vej 23},
city={Aalborg {\O}st},
postcode={9220},
country={Denmark}}

\affiliation[2]{organization={Institute of Computing, Faculty of Informatics, Università della Svizzera italiana},
addressline={Via la Santa 1},
city={Lugano},
postcode={6900},
country={Switzerland}}

\begin{abstract}
A nonlinear generalisation of the PageRank problem involving the Moore-Penrose inverse of an incidence matrix is developed for local graph partitioning purposes. The Levenberg-Marquardt method with a full rank Jacobian variant provides a strategy for obtaining a numerical solution to the generalised problem. Sets of vertices are formed according to the ranking supplied by the solution, and a conductance criterion decides upon the set that best represents the cluster around a starting vertex. Experiments on both synthetic and real-world inspired graphs demonstrate the capability of the approach to not only produce low conductance sets, but to also recover local clusters with an accuracy that consistently surpasses state-of-the-art algorithms. 
\end{abstract}


\end{frontmatter}

\section{Introduction}
Graphs are ubiquitous as a means of representing objects and their relationships. It is the very abstraction inherent in the definition of a graph that makes them so widely applicable. Vertices can be any kind of object --- such as images~\cite{floros2022fast} or geographical locations~\cite{pasadakis2023} --- and there are no restrictions on the number of connections between them, which are expressed as edges.

A graph may well exhibit a granular structure. Vertices that share a common property can be uncovered by a process called \emph{clustering}. The resulting distinct groupings of vertices are referred to as \emph{communities} or \emph{clusters}. For example, if all connections are equally valuable, then it may be that communities are distinguished by the relatively high number of intra-community edges compared to the inter-community number of edges.

\emph{Spectral clustering} is a prominent algorithm that casts clustering as a graph partitioning problem. Clusters are often determined through an analysis of the combinatorial Laplacian matrix eigenvectors~\cite{ng2001spectral,von2007tutorial}. A generalisation of spectral clustering utilising the $p$-Laplacian~\cite{amghibech2003eigenvalues,tudisco2018nodal} was pursued in~\cite{buhler2009spectral}. Improved clustering assignments were demonstrated in the bi-partitioning case. Recursive bisection is, however, required to obtain a higher number of clusters. Direct multiway approaches include approximating the $p$-orthogonality constraint~\cite{luo2010eigenvectors,pasadakis2022multiway}, the effective $p$-resistance~\cite{saito2023multi}, and exploiting concepts from total variation used commonly in image processing~\cite{bresson2013multiclass}.

When only a single cluster around a vertex or vertices is of interest, then an alternative strategy can be employed. \emph{Local clustering}, as the name suggests, seeks to leverage local structure and information to identify a single cluster. This has the potential to improve algorithmic run-time by avoiding computations over the entirety of a graph.

Treatment of clustering as a partitioning problem carries over into local clustering. A cut can be found using a variation of the PageRank problem~\cite{andersen2006}. Since a PageRank vector provides a ranking of vertices (refer to~\cite{kollias2014} for information on functional rankings), sets of vertices can then be assembled according to their order and interrogated to reveal a cut. Similarly to spectral clustering, algorithms based on the $p$-norm have been introduced for local clustering purposes. There is the nonlinear $p$-norm cut algorithm~\cite{liu2020strongly}, and an algorithm based on the idea of diffusion with $p$-norm network flow~\cite{fountoulakis2020p}.

This work takes the system of linear equations that forms the PageRank problem discussed in Section~\ref{sec:pr} as the starting point in the construction of a local clustering algorithm. Only clusters around a single vertex are considered. A nonlinear generalisation of the PageRank problem defined on a simple, connected and weighted graph that is inspired by the $p$-norm is proposed in Section~\ref{sec:npr}. The Moore-Penrose inverse of the incidence matrix plays a very important role in the generalisation. It is shown in Section~\ref{sec:npr} that the generalised problem reduces to that of the (linear) PageRank problem in the limit as the number of vertices tends to infinity. Additionally, in Section~\ref{sec:npr}, an infinitesimal perturbation argument offers an insight into the positive effect of the generalised problem on the cluster criterion. The Levenberg-Marquardt method with a full rank Jacobian variant for obtaining a numerical solution to the generalised problem is outlined in Section~\ref{sec:numerics}.

The next section presents the theory required for the development of the generalised problem and the identification of a cluster. Section~\ref{sec:exp} highlights the capability of the proposed algorithm on a number of synthetic and real-world inspired graphs. The algorithm performs strongly on a consistent basis, achieving results better than state-of-the-art algorithms.

\section{Preliminaries}
\label{sec:prelim}
\subsection{Notions of graph theory}
\label{subsec:notions}
A \emph{graph} $\mathcal{G} = (\mathcal{V}, \mathcal{E})$ consists of a finite set $\mathcal{V}$ of \emph{vertices} and a set of unordered non-repeating pairs of distinct elements of $\mathcal{V}$ called \emph{edges}, $\mathcal{E} \subseteq \mathcal{V} \times \mathcal{V}$. Any two vertices $u, v \in \mathcal{V}$ are said to be \emph{adjacent} or \emph{neighbours} if they form an edge $e = \{ u, v \}$ of $\mathcal{E}$. Their relationship is indicated by $u \sim v$. By definition, there can be no loops, i.e., $\{ v , v \} \notin \mathcal{E}$, and only one edge can ever join two vertices. Such a graph is often described as being \emph{simple}.

If a \emph{path} exists between every pair of vertices $u$ and $v$ in $\mathcal{G}$, that is, a sequence of vertices $u = u_1, u_2, \ldots, u_n, u_{n+1} = v$ with $e_i = \{ u_i, u_{i+1} \} \in \mathcal{E}$ for $i = 1,2, \ldots, n$, then the graph is said to be \emph{connected}.

A \emph{weighted graph} is a graph $\mathcal{G} = (\mathcal{V},\mathcal{E})$ with an associated \emph{weight} function $w : \mathcal{V} \times \mathcal{V} \to [0,\infty)$, commonly expressed as the triple $\mathcal{G} = (\mathcal{V},\mathcal{E},w)$, satisfying $w(v,v) = 0$ if $v \in \mathcal{V}$, $w(u,v) = w(v,u)$ if $u \sim v$, and $w(u,v) = 0$ if and only if $u \not\sim v$. Furthermore, $w(u,v) > 0$ for all $\{ u, v \} \in \mathcal{E}$. Hereafter, only simple, connected and weighted graphs will be considered. A graph $\mathcal{G}$ should be read in this context.

\begin{remark}
	A weight function $w(u,v) = 1$ for all $\{ u, v \} \in \mathcal{E}$ represents the non-weighted graph equivalent.
\end{remark}

Let $\mathcal{H} (\mathcal{V})$ denote the Hilbert space of real-valued functions defined on the vertices of a graph $\mathcal{G}$. Elements of $\mathcal{H} (\mathcal{V})$ assign a real-value, say $x(v)$, to each vertex $v \in \mathcal{V}$. The standard inner product $\langle x, y \rangle = \sum_{v \in V} x(v) y(v)$, where $x, y \in \mathcal{H} (\mathcal{V})$, will be assumed. For $x \in \mathcal{H} (\mathcal{V})$, the norm is then given by $\lVert x \rVert = \sqrt{\langle x, x \rangle}$. $x$ can be treated as the column vector $x = \left( x(v_1), x(v_2), \ldots, x(v_{|\mathcal{V}|}) \right)^\top$ in $\mathbb{R}^{|\mathcal{V}|}$.

\subsection{Conductance and sweep-cut}
\label{sec:conductance}
For a subset $\mathcal{S}$ of the vertices in a graph $\mathcal{G} = (\mathcal{V},\mathcal{E},w)$, i.e., $\mathcal{S} \subseteq \mathcal{V}$, the \emph{volume} of $\mathcal{S}$ has the form
\begin{equation*}
	\text{vol} \, (\mathcal{S}) = \sum_{v \in \mathcal{S}} d (v) ,
\end{equation*}
in which $d (v)$ stands for the (weighted) \emph{degree} of vertex $v \in \mathcal{V}$:
\begin{equation*}
d (v) = \sum_{u \sim v} w(u,v) .
\end{equation*} 
The complement of $\mathcal{S}$ is $\bar{\mathcal{S}} = \mathcal{V} \setminus \mathcal{S}$. When $\mathcal{S} \subset \mathcal{V}$, an \emph{edge boundary} $\partial (\mathcal{S})$ of $\mathcal{S}$ collects the edges with one vertex in $\mathcal{S}$ and the other in $\bar{\mathcal{S}}$, that is,
\begin{equation*}
	\partial (\mathcal{S}) = \left\{ \{ u, v \} \in \mathcal{E} \mid u \in \mathcal{S} \mbox{ and } v \in \bar{\mathcal{S}} \right\} .
\end{equation*}
\emph{Conductance} (or the \emph{Cheeger ratio}) of $\mathcal{S}$ is defined as
\begin{equation*}
	\Phi (\mathcal{S}) = \frac{ w \left( \partial (\mathcal{S}) \right) }{\min \left( \text{vol} \, (\mathcal{S}), \text{vol} \, (\bar{\mathcal{S}}) \right)} ,
\end{equation*}
where $w \left( \partial (\mathcal{S}) \right) = \sum_{e \in \partial (\mathcal{S})} w (e)$.

A graph can be partitioned into two based on the conductance. This can be accomplished through a \emph{sweep} over a vector $x \in \mathcal{H} (\mathcal{V})$ for cut (edges), $\partial (\mathcal{S})$, discovery. Suppose $v_1, \ldots, v_{|\mathcal{V}|}$ is an ordering of vertices such that $x(v_i) \geq x(v_{i+1})$ for $i = 1, \ldots, |\mathcal{V}| - 1$. What are known as \emph{sweep sets} $\mathcal{S}_j = \{ v_1, \ldots, v_j \}$ for all $j = 1, \ldots, |\mathcal{V}| - 1$ can then be formed. Let $\mathcal{N} = \{ S_1, \ldots, S_{|\mathcal{V}| - 1} \}$. It is the smallest conductance
\begin{equation*}
	\mathcal{S}_* = \arg \, \min_{ \mathcal{S} \in \mathcal{N} } \, \Phi ( \mathcal{S} )
\end{equation*}
that provides the partition.

\subsection{Matrix representation of graphs}
Structural information of a graph $\mathcal{G} = (\mathcal{V},\mathcal{E},w)$ can be encoded in matrix form. Let $n = |\mathcal{V}|$ and $m = |\mathcal{E}|$. Connectivity between vertices is captured in the $n \times n$ symmetric \emph{adjacency matrix} $A$, which has the entries
\begin{equation*}
	a_{ij} = \begin{cases}
 	w \left( v_i, v_j \right), &\mbox{ if } v_i \sim v_j,\\
 	0, &\mbox{ otherwise} .
 \end{cases}
\end{equation*}
The \emph{degree matrix} $D$ is the diagonal matrix with $d_{ii} = \sum_{j=1}^{n} a_{ij} = d (v_i)$ for $v_i \in V$ and $i = 1, 2, \ldots, n$. A (random walk) \emph{transition probability matrix} can then be defined as $P = D^{-1} A$.

Edge-vertex connectivity features in the $m \times n$ \emph{incidence matrix} $B$. For the sole purpose of the definition, an orientation on the graph is assumed. This entails the specification of an arbitrary but fixed order to the vertices of every edge in $\mathcal{E}$. An edge with ordered vertices is written $\vec{e} = [u, v]$, in which $u = o(\vec{e})$ is the \emph{origin} vertex and $v = t(\vec{e})$ the \emph{terminus} vertex. Entries of the incidence matrix are then
\begin{equation*}
	b_{ij} = \begin{cases}
		1, &\mbox{ if $v_j$ is the terminus vertex of $\vec{e}_i$, i.e., $v_j = t(\vec{e}_i)$},\\
		-1, &\mbox{ if $v_j$ is the origin vertex of $\vec{e}_i$, i.e., $v_j = o(\vec{e}_i)$},\\
		0, &\mbox{ otherwise} .
	\end{cases}
\end{equation*}
As long as the graph is simple and connected, there are no all-zero columns and exactly two non-zero entries in every row.


\begin{proposition}\label{prop:rnkB}
	For a simple and connected graph with $n$ vertices, $\text{rank} \, (B) = n - 1$.
\end{proposition}

\begin{proof}	
	Based on~\cite[Lemma 2.2 on p.~12]{bapat2010graphs}. The following are a result of there being only two non-zero entries, namely $-1$ and $1$, in every row of $B$.
	\renewcommand{\theenumi}{\roman{enumi}}
	\begin{enumerate}
		\item As each row of $B$ sums to zero, the columns of $B$ are linearly dependent and $\text{rank} \, (B) < n$. Thus, $\text{rank} \, (B) \leq n - 1$.
		\item Consider $B x = 0$. The entries of $x$ must all be equal due to the connectedness of the graph. As such,  $\text{rank} \, (B)$ is at least $n - 1$.
	\end{enumerate}
	In conclusion, $\text{rank} \, (B) = n - 1$.
\end{proof}
An $m \times m$ matrix $C$ with edge weights on the diagonal, i.e., $C = \text{diag} \left( w(e_1), w(e_2), \ldots, w(e_m) \right)$, is often a useful accompaniment to the incidence matrix.


The combinatorial \emph{Laplacian matrix} can be defined either in terms of the incidence matrix or the adjacency matrix by
\begin{equation*}
	L = B^\top C B = D - A .
\end{equation*}
Note that $L$ is symmetric, positive semidefinite and singular~\cite[Proposition~3.4 on p.~1193]{kodsi2021truss}. 


\begin{proposition}\label{prop:lapsing}
	For a simple, connected and weighted graph with $n$ vertices, $\text{rank} \, (L) = n - 1$.
\end{proposition}

\begin{proof}
	It is shown that $\text{ker} \, (L) = \text{ker} \, (B)$ and the rank is provided by Proposition~\ref{prop:rnkB}. If $x \in \text{ker} \, (L)$, then $L x = 0$. Consider $x^\top L x = (B x)^\top C (B x) = 0$. Since $C$ is positive definite, it follows that $B x = 0$ and $x \in \text{ker} \, (B)$. Conversely, if $x \in \text{ker} \, (B)$, then $B x = 0$. Thus, $L x = 0$ and $x \in \text{ker} \, (L)$.
\end{proof}


\subsection{Moore-Penrose inverse}
Moore in 1920~\cite{dresden1920fourteenth} (at the fourteenth western meeting of the American Mathematical Society) presented an extension to the notion of a nonsingular square matrix inverse that covers (finite-dimensional) rectangular matrices. Even though this extension featured in~\cite{moore1935general}, not much attention seems to have been paid to it. This could possibly be due to the rather unfortunate notation usage~\cite{benisrael1986,campbell2009}. It was not until 1955 that Penrose~\cite{penrose1955} was to put forward an equivalent theory~\cite{campbell2009}. Let $Y$ be an $m \times n$ matrix. In Penrose's approach, $Y^\dag$ is the unique $n \times m$ matrix that satisfies
\begin{subequations}
\begin{align}
	Y Y^\dag Y &= Y , \label{eq:pen1} \\
	Y^\dag Y Y^\dag &= Y^\dag , \label{eq:pen2} \\
	\left( Y Y^\dag \right)^* &= Y Y^\dag , \label{eq:pen3} \\
	\left( Y^\dag Y \right)^* &= Y^\dag Y , \label{eq:pen4}
\end{align}
\end{subequations}
where $Y^*$ represents the conjugate transpose of $Y$. Note that $\text{ker} \, (Y^\dag) = \text{ker} \, (Y^*)$ \cite[Theorem~1.2.2 on p.~12]{campbell2009}. When $Y$ is real, so is $Y^\dag$, due to the uniqueness of the solution. If $Y$ is nonsingular, then $Y^\dag$ reduces to the familiar inverse of a square matrix, i.e., $Y^{-1}$.

An explicit formula for the Moore-Penrose inverse of $Y$ with full (column) rank can be obtained from~\cite[Corollary~1 on p.~674]{pearl1966generalized} by setting $Y = Y I$, in which $I$ is the standard identity matrix:
\begin{equation*}
	Y^\dag = \left( Y^* Y \right)^{-1} Y^*, 
\end{equation*}
or $Y^\dag = \left( Y^\top Y \right)^{-1} Y^\top$, if $Y$ is real. A result that will prove to be useful is recorded in the following proposition.

\begin{proposition}\label{prop:mpi}
	For an $m \times n$	matrix $Y$, $(Y^\dag Y)^\dag = Y^\dag Y$.
\end{proposition}

\begin{proof}
	Let $Z = Y^\dag Y$. It follows that $Z Z = (Y^\dag Y Y^\dag) Y = Y^\dag Y = Z$ and $Z^* = (Y^\dag Y)^* = Y^\dag Y = Z$  by (\ref{eq:pen2}) and (\ref{eq:pen4}), respectively. Say $X = Z$. It is now shown that $X$ satisfies~(\ref{eq:pen1})--(\ref{eq:pen4}):
	\begin{align*}
		Z X Z = Z Z Z = Z Z &= Z,\\
		X Z X = Z Z Z = Z &= X,\\
		\left( Z X \right)^* = \left( Z Z \right)^* = Z^* = Z = ZZ &= Z X, \\
		\left( X Z \right)^* = \left( Z Z \right)^* = Z^* = Z = ZZ &= X Z .
	\end{align*}
	Thus, $X = Z^\dag$.
\end{proof}

\subsection{Notation}
\begin{itemize}
	\item $q, Q$ represent an all-ones vector and matrix, respectively.
	\item $Y \geqslant 0$ denotes a matrix with non-negative entries and that $Y \neq 0$.
	\item The \emph{Hadamard product} of $m \times n$ real matrices $Y$ and $Z$ is the entry-wise product matrix $Y \odot Z = \left( y_{ij} z_{ij} \right)$ of the same size.
	\item The \emph{Hadamard power} of an $m \times n$ real matrix $Y$ with respect to a positive integer $p \in \mathbb{Z}_{> 0}$ is of the form $Y^{\circ p} = \left( y_{ij}^p \right)$. $Y$ with only positive entries allows for $p \in \mathbb{R}$.
	\item For $x \in \mathcal{H} (\mathcal{V})$ and $\mathcal{S} \subseteq \mathcal{V}$, $x[\mathcal{S}] = \sum_{v \in \mathcal{S}} x(v)$.
\end{itemize}

\section{PageRank problem}
\label{sec:pr}
Consider $\mathcal{G} = (\mathcal{V},\mathcal{E},w)$ with $\mathcal{E} \neq \emptyset$. Let $n = |\mathcal{V}|$ and $m = |\mathcal{E}|$. Assume a probability vector $r \in \mathcal{H} (\mathcal{V})$, known as the \emph{teleportation vector}, in which there will only be a single non-zero entry for a \emph{starting vertex} $s \in \mathcal{V}$, i.e., $r(s) = 1$ and $r(v) = 0$ for all $v \in \mathcal{V} \setminus \{ s \}$. Given a \emph{damping factor} $\alpha \in (0,1)$, the PageRank vector $x \in \mathcal{H} (\mathcal{V})$ is the solution of the eigenvector problem
\begin{equation}\label{eq:eigpr}
	U x = x \mbox{ with } q^\top x = 1 ,
\end{equation}
where $U = \alpha P^\top + ( 1 - \alpha ) r q^\top$. Since $U$ is a stochastic matrix, it follows that the largest eigenvalue is equal to $1$. The PageRank vector, thus, corresponds to this eigenvalue and is a probability vector, as $q^\top x = 1$.

An equivalent formulation of problem~(\ref{eq:eigpr}) can be written in terms of the system of linear equations
\begin{equation*}
	\left( I - \alpha P^\top \right) x = ( 1 - \alpha ) r .
\end{equation*}
Introduction of the transition probability matrix definition along with $A = D - L$ returns
\begin{equation}\label{eq:lpr}
	\left( \beta I + L D^{-1} \right) x = \beta r ,
\end{equation}
where $\beta = (1-\alpha) \big/ \alpha$.


\begin{proposition}\label{prop:lprexists}
		Let $T = \beta I + L D^{-1}$ with $\beta \in (0,\infty)$. $T^{-1}$ exists, and $T^{-1} \geqslant 0$.
\end{proposition}

\begin{proof}
	Since all the entries on the main diagonal of $T$ are positive and $T D = \beta D + L = (\beta + 1) D - A$ is strictly diagonally dominant, i.e.,
	\begin{displaymath}
		(\beta + 1) d_{ii} > \sum_{j = 1, j \neq i}^n | a_{ij} |
	\end{displaymath}
	for $i = 1, 2 , \ldots, n$, it follows that $T$ is a nonsingular $M$-matrix according to~\cite[Condition~N\textsubscript{39} on p.~182]{plemmons1977m}. A notable characteristic of a nonsingular $M$-matrix is that the entries of the inverse are all non-negative. Subsequently, $T^{-1} \geqslant 0$.
\end{proof}


A perspective of $T : \mathcal{H} (\mathcal{V}) \to \mathcal{H} (\mathcal{V})$ as an operator can be taken, which satisfies
\begin{equation*}
	\left( T x \right) (v) = \left( \beta + 1 \right) x(v) - \sum_{u \sim v} \frac{w(u,v)}{d(u)} x(u)
\end{equation*}
for each $v \in \mathcal{V}$. Define
\begin{equation*}
	\hat{x} (v) = \begin{cases}
	\left( \beta + 1 \right) x(v), & \mbox{if $v \in \mathcal{V} \setminus \{ s \}$} ,\\
	\left( \beta + 1 \right) x(v) - \beta, & \mbox{if $v=s$} .
	\end{cases}	
\end{equation*}
Then, for all $v \in \mathcal{V}$,
\begin{equation*}
  	\hat{x} (v) = \sum_{u \sim v} \frac{w(u,v)}{d(u)} x(u) .
\end{equation*}


\begin{proposition}\label{prop:max}
	Let $\mathcal{G} = (\mathcal{V},\mathcal{E},w)$ with $\mathcal{E} \neq \emptyset$. Assume $\beta \in (0,\infty)$. Given $\left( T x \right) (v) = \beta r (v)$, it follows that $x(v) < 1$ for all $v \in \mathcal{V}$.
\end{proposition}

\begin{proof}
A summation of $x (v)$ over all $v \in \mathcal{V}$ yields
\begin{align*}
	x [ \mathcal{V} ] &= \frac{1}{\beta + 1} \sum_{v \in \mathcal{V}} \sum_{u \sim v} \frac{w(u,v)}{d(u)} x(u) + \frac{\beta}{\beta + 1} \\
	&= \frac{1}{\beta + 1} \sum_{u \in \mathcal{V}} \frac{x(u)}{d(u)} \sum_{v \sim u} w(u,v) + \frac{\beta}{\beta + 1} \\
	&= \frac{1}{\beta + 1} x [ \mathcal{V} ] + \frac{\beta}{\beta + 1} .
\end{align*}
Thus, $x [ \mathcal{V} ] = 1$. Note that $x(v) \geq 0$ for all $v \in \mathcal{V}$, as $T^\text{-1} \geqslant 0$ (see Proposition~\ref{prop:lprexists}) and $r(v) \geq 0$ for all $v \in \mathcal{V}$. If $x(v) = 1$ and $\hat{x} (v) = \beta + 1$ for any $v \in \mathcal{V} \setminus \{ s\}$, then by implication $\beta + 1 = 0$. This contradicts $\beta > 0$. Similarly, $x(s) = 1$ does not hold. As such, $x(v) < 1$ for all $v \in \mathcal{V}$.	
\end{proof}


\begin{proposition}\label{prop:ineq}
	Further to that specified in Proposition~\ref{prop:max}, take $d(v) \geq 1$ for all $v \in \mathcal{V}$. Let $\mathcal{F}$ be a proper subset of $\mathcal{V}$ with $s \in \mathcal{F}$. Suppose there exists $\vartheta > 0$ and $x(u) \leq \vartheta$ for all $ u \in \bar{\mathcal{F}}$. If
	\begin{equation*}
		2 \beta \vartheta \sum_{\substack{u \sim v\\ u, v \in \bar{\mathcal{F}}}} w(u,v) \leq w \left( \partial (\mathcal{F}) \right) ,
	\end{equation*}
	then
	\begin{equation*}
		x [ \bar{\mathcal{F}} ] < \frac{1}{\beta} w \left( \partial (\mathcal{F}) \right) .
	\end{equation*}	
\end{proposition}

\begin{proof}
	A summation of $x(v)$ over all $v \in \bar{ \mathcal{F} }$ equates to
	\begin{align*}
		x [ \bar{\mathcal{F}} ] &= \sum_{v \in \bar{\mathcal{F}}} \frac{1}{(\beta + 1)} \sum_{u \sim v} \frac{w(u,v)}{d(u)} x(u) \\
		&= \frac{1}{\beta + 1} \left( \sum_{\substack{u \sim v\\ u, v \in \bar{\mathcal{F}}}} w(u,v) \left( \frac{x(u)}{d(u)} + \frac{x(v)}{d(v)}  \right) + \sum_{\substack{u \sim v\\ u \in \mathcal{F}, v \in \bar{\mathcal{F}}}} \frac{w(u,v)}{d(u)} x(u) \right) .
	\end{align*}
	Since $x(u) \leq \vartheta$ for all $u \in \bar{\mathcal{F}}$ and $d(v) \geq 1$ for all $v \in \mathcal{V}$, it follows that
	\begin{displaymath}
		x [ \bar{\mathcal{F}} ] \leq \frac{1}{\beta + 1} \left( 2 \vartheta \sum_{\substack{u \sim v\\ u, v \in \bar{\mathcal{F}}}} w(u,v) + \sum_{\substack{u \sim v\\ u \in \mathcal{F}, v \in \bar{\mathcal{F}}}} w(u,v) x(u) \right) .
	\end{displaymath}
	Based on Proposition~\ref{prop:max},
	\begin{displaymath}
		x [ \bar{\mathcal{F}} ] < \frac{1}{\beta + 1} \left( 2 \vartheta \sum_{\substack{u \sim v\\ u, v \in \bar{\mathcal{F}}}} w(u,v) +  w \left( \partial (\mathcal{F}) \right) \right) .
	\end{displaymath} 
	Substitution of the first term on the right-hand side with $w \left( \partial (\mathcal{F}) \right) / \beta$ leads to
	\begin{displaymath}
		x [ \bar{\mathcal{F}} ] < \frac{1}{\beta + 1} \left( \frac{1}{\beta} w \left( \partial (\mathcal{F}) \right) +  w \left( \partial (\mathcal{F}) \right) \right) = \frac{1}{\beta} w \left( \partial (\mathcal{F}) \right) . \qedhere
	\end{displaymath}
\end{proof}
This demonstrates that for any proper subset $\mathcal{F}$ containing $s$ and satisfying the stated condition, the cut weight scaled by $1/\beta$ bounds the sum value of PageRank vector entries corresponding to vertices contained in the complement $\bar{\mathcal{F}}$.

\section{Nonlinear PageRank problem}
\label{sec:npr}
A modification involving the replacement of $x$ with $f(x)$ in the system of linear equations~(\ref{eq:lpr}), i.e.,
\begin{equation*}
	T f(x) = \beta r ,
\end{equation*}
is proposed for local graph partitioning purposes. The question is what form should $f$ take.


\begin{proposition}
	Let $p \in (1, \infty)$. Then 
	\begin{equation*}
		\frac{\partial}{\partial x} \left( \lVert B x \rVert_p^p  \right) = p \left( | B x |^{\circ (p-2)} \odot Bx \right)^\top B .
	\end{equation*}	
\end{proposition}

\begin{proof}
	Suppose $z = Bx$. All that is required is
	\begin{align*}
		\frac{\partial}{\partial x_j} \left( \lVert Bx \rVert_p^p  \right) &= \frac{\partial}{\partial z_i} \left( \lVert z \rVert_p^p  \right) \frac{\partial z_i}{\partial x_j} \\
		&=  \left( \sum_{k=1}^m \frac{\partial }{\partial z_i} |z_k|^p \right)  \frac{\partial z_i}{\partial x_j} \nonumber \\
		&= \left( \sum_{k=1}^m p |z_k|^{p-1} \frac{\partial}{\partial z_i} |z_k| \right) b_{ij} \nonumber \\
		&= p |z_i|^{p-2} z_i b_{ij} ,
	\end{align*}
	for $j = 1, 2, \ldots, n$, where summation over the repeated index is assumed.
\end{proof}


If $\frac{\partial}{\partial x} \left( \lVert B x \rVert_2^2 \right) = 0$, then $B^\top B x = 0$. Observe that, by~(\ref{eq:pen1}),
\begin{equation*}
	B^\top B x = L x = B^\top B B^\dag B x = L B^\dag B x
\end{equation*}
for a non-weighted graph. This inspires the choice of
\begin{equation*}
	f(x) = B^\dag \left( \left( (Bx)^{\circ 2} + \zeta q \right)^{\circ \frac{1}{2} \left( p - 2 \right)} \odot Bx \right) ,
\end{equation*}
in which $\zeta \in \mathbb{R}_{> 0}$ represents a small value. What will be referred to as the \emph{Nonlinear PageRank problem} can now be written in full as
\begin{equation*}
	T B^\dag \left( \left( (Bx)^{\circ 2} + \zeta q \right)^{\circ \frac{1}{2} \left( p - 2 \right)} \odot Bx \right) = \beta r ,
\end{equation*}
where $\beta \in (0,1)$ and $p \in (1,2]$.
\begin{remark}
	Properties of the Moore-Penrose inverse of an arbitrary incidence matrix are examined in~\cite{ijiri1965generalized}.
\end{remark}


The solution to the Nonlinear PageRank problem for $p = 2$ can be expressed as an additive combination of the corresponding (linear) PageRank problem solution and a constant term. Before this can be demonstrated, the following proposition must first be established.


\begin{proposition}\label{prop:mp}
	For a simple and connected graph with $n$ vertices, $I - B^\dag B = \frac{1}{n} Q$.
\end{proposition}

\begin{proof}
	Based on~\cite[Theorem~1 on p.~829]{ijiri1965generalized}. It follows from~(\ref{eq:pen1}) that
	\begin{equation}\label{eq:incidenceprop}
		B \left( I - B^\dag B \right) = 0 .
	\end{equation}
	Since each row of $B$ has only two non-zero entries, those being $-1$ and $1$, the entries in each column of $\left( I - B^\dag B \right)$ have to be equal for~(\ref{eq:incidenceprop}) to hold.
	
	Note that $\left( I - B^\dag B \right)$ fulfils the conditions of an orthogonal projection matrix, namely symmetry and idempotency. Symmetry is possible because of~(\ref{eq:pen4}) and is the reason why all the entries of $\left( I - B^\dag B \right)$ are equal. Because of the idempotent property, i.e.,
	 \begin{displaymath}
	 	\left( I - B^\dag B \right)^2 = I - 2 B^\dag B  + \left( B^\dag B B^\dag \right) B = I - 2 B^\dag B  + B^\dag B = I - B^\dag B ,
	 \end{displaymath}
	the entries of $\left( I - B^\dag B \right)$ are all equal to $1 \big/ n$. Otherwise, $\left( I - B^\dag B \right)^2 \neq I - B^\dag B$. 
\end{proof}


\begin{proposition}\label{prop:p2equiv}
	Let $\mathcal{G} = (\mathcal{V}, \mathcal{E}, w)$ with $\mathcal{E} \neq \emptyset$ and $n = | \mathcal{V} |$. For $p = 2$, the solution in the least squares sense to the Nonlinear PageRank problem defined on $\mathcal{G}$ has the form
	\begin{equation*}
		x = c - \frac{1}{n} q ,
	\end{equation*}
	where $c = \beta T^{-1} r$ is the solution of the corresponding (linear) PageRank problem.
\end{proposition}

\begin{proof}
	The Nonlinear PageRank problem for $p = 2$ defined on $\mathcal{G}$ with $x \in \mathcal{H} (\mathcal{V})$ equates to
	\begin{displaymath}
		T B^\dag B x = \beta r .
	\end{displaymath}
	Proposition~\ref{prop:lprexists} discloses the nonsingular nature of $T$, which supports $B^\dag B x = c$. By Proposition~\ref{prop:mpi}, the minimal least squares solution then is
	\begin{displaymath}
		x = B^\dag B c .
	\end{displaymath}
	This can be re-expressed as 
	\begin{displaymath}
		x = \left(I - \frac{1}{n} Q \right) c = c - \frac{1}{n} q q^\top c
	\end{displaymath}	
	courtesy of Proposition~\ref{prop:mp}. Recall that $\alpha = 1 / (1 + \beta)$ and $\beta \in (0,1)$. As such, $q^\top c = 1$. Thus, $x = c - (1 / n) q$.
\end{proof}


\begin{remark}
Clearly, $\lim_{n \to \infty} \lVert (1/n) q \rVert = \lim_{n \to \infty} 1 / \sqrt{n} = 0$. 
\end{remark}

In order to gain an insight into the effect that the Nonlinear PageRank problem or rather the solution has on the conductance, a perturbation argument is developed.


\begin{proposition}\label{prop:perturb}
	Let $\mathcal{G} = (\mathcal{V}, \mathcal{E}, w)$ with  $\mathcal{E} \neq \emptyset$ and $n = | \mathcal{V} |$. Consider the case where $0 < w_{\min} \leq w (v_i, v_j) \leq w_{\max} < \infty$ for all $\{ v_i, v_j \} \in \mathcal{E}$ and every vertex $v \in \mathcal{V}$ has at most $\Delta$ neighbours. $x^{0}$ is to be the least squares solution of the Nonlinear PageRank problem defined on $\mathcal{G}$ at $p = 2$. Take $\epsilon \in \mathbb{R}_{\geq 0}$ to be a small value. For sufficiently large (but finite) $n$, there exists $x (\epsilon) = x^{0} + \epsilon x^{1} + \mathcal{O} (\epsilon^2)$ such that the Perturbed Nonlinear PageRank problem
	\begin{equation}\label{eq:pnpr}
		T B^\dag \left( \left( \left( B x^{0} \right)^{\circ 2} + \zeta q \right)^{\circ - \frac{\epsilon}{2}} \odot B x(\epsilon) \right) = \beta r
	\end{equation}
	satisfies
	\begin{equation*}
		\left\lVert \beta r - T B^\dag \left( \left( \left( B x^{0} \right)^{\circ 2} + \zeta q \right)^{\circ - \frac{\epsilon}{2}} \odot B x(\epsilon) \right) \right\rVert \leq \frac{h_1}{\sqrt{n}} + \frac{h_2}{\sqrt{n}} \epsilon + h_3 \epsilon^2 ,
	\end{equation*}
	in which $h_1 = (\beta + 1) + \sqrt{(w_{max} / w_{min}) \Delta}$ and $h_2, h_3 > 0$ are constants independent of both $\epsilon$ and $n$.
\end{proposition}

\begin{proof}
	A Taylor series expansion of $\left( \left( B x^{0} \right)^{\circ 2} + \zeta q \right)^{\circ - \frac{\epsilon}{2}}$ truncated at the second order yields
	\begin{displaymath}
		\left( \left( B x^{0} \right)^{\circ 2} + \zeta q \right)^{\circ - \frac{\epsilon}{2}} = q - \frac{\epsilon}{2} y + \mathcal{O} (\epsilon^2) ,
	\end{displaymath}
	where $y = \ln \left( \left( B x^{0} \right)^{\circ 2} + \zeta q \right)$.
	
	Since $x (\epsilon) = x^{0} + \epsilon x^{1} + \mathcal{O} (\epsilon^2)$, it follows that
	\begin{align*}
		\left( \left( B x^{0} \right)^{\circ 2} + \zeta q \right)^{\circ - \frac{\epsilon}{2}} \odot B x(\epsilon) &= \left( q - \frac{\epsilon}{2} y + \mathcal{O} (\epsilon^2) \right) \odot \left( B x^{0} + \epsilon B x^{1} + \mathcal{O} (\epsilon^2) \right)\\
		&= B x^{0} + \epsilon B x^{1} - \frac{\epsilon}{2} \left( y \odot B x^{0} \right) + \mathcal{O} (\epsilon^2)
	\end{align*}
	and
	\begin{multline*}
		T B^\dag \left( \left( \left( B x^{0} \right)^{\circ 2} + \zeta q \right)^{\circ - \frac{\epsilon}{2}} \odot B x(\epsilon) \right) = T B^\dag B x^{0} + \epsilon T B^\dag B x^{1} \\ - \frac{\epsilon}{2} T B^\dag \left( y \odot B x^{0} \right) + \mathcal{O} (\epsilon^2) .
	\end{multline*}
	By Propositions~\ref{prop:mp} and~\ref{prop:p2equiv},
	\begin{align*}
		T B^\dag B x^{0} &= T \left( I - \frac{1}{n} q q^\top \right) \left( c - \frac{1}{n} q \right) \\
		&= T \left( c - \frac{1}{n} q q^\top c - \frac{1}{n} q + \frac{1}{n^2} q q^\top q \right) \\
		& = T c - \frac{1}{n} T q ,
	\end{align*}
	as $q^\top c = 1$. But $c = \beta T^{-1} r$, which leads to $T B^\dag B x^{0} = \beta r - (1 / n) T q$.
	
	It is now possible to re-express the Perturbed Nonlinear PageRank problem in the form
	\begin{displaymath}
		\beta r = \beta r - \frac{1}{n} T q + \epsilon T B^\dag B x^{1} - \frac{\epsilon}{2} T B^\dag \left( y \odot B x^{0} \right) + \mathcal{O} (\epsilon^2) .
	\end{displaymath}
	Set $x^{1} = (1 / 2) B^\dag \left(y \odot B x^0 \right)$ so that $T B^\dag B x^{1} = (1 / 2) T B^\dag \left( y \odot B x^{0} \right)$. Then
	\begin{displaymath}
		 \left\lVert \beta r - T B^\dag \left( \left( \left(B x^{0} \right)^{\circ 2} + \zeta q \right)^{\circ - \frac{\epsilon}{2}} \odot B x(\epsilon) \right) \right\rVert \leq \frac{1}{n} \left\lVert T q \right\rVert + \left\lVert \mathcal{O} (\epsilon^2) \right\rVert .
	\end{displaymath}
	Recall that $T = \beta I + L D^{-1} = \beta I + (D - A) D^{-1} = (\beta + 1) I - A D^{-1}$. This leads to
	\begin{displaymath}
		\left\lVert T q \right\rVert \leq \left\lVert T \right\rVert \left\lVert q \right\rVert = \left\lVert (\beta + 1) I - A D^{-1} \right\rVert \sqrt{n} \leq (\beta + 1) \sqrt{n} + \left\lVert A D^{-1} \right\rVert \sqrt{n} .
	\end{displaymath}
	Based on $\left\lVert A D^{-1} \right\rVert \leq \left\lVert A D^{-1} \right\rVert_{1}^{\frac{1}{2}} \left\lVert A D^{-1} \right\rVert_{\infty}^{\frac{1}{2}} = \sqrt{(w_{max} / w_{min}) \Delta}$,
	\begin{displaymath}
		\frac{1}{n} \left\lVert T q \right\rVert \leq \frac{1}{\sqrt{n}} (\beta + 1) + \frac{1}{\sqrt{n}} \left( \frac{w_{max}}{w_{min}} \Delta \right)^{\frac{1}{2}} = \frac{1}{\sqrt{n}} \left( (\beta + 1) + \left( \frac{w_{max}}{w_{min}} \Delta \right)^{\frac{1}{2}} \right) .
	\end{displaymath}
	Finally,
	\begin{displaymath}
		\left\lVert \beta r - T B^\dag \left( \left( \left( B x^{0} \right)^{\circ 2} + \zeta q \right)^{\circ - \frac{\epsilon}{2}} \odot B x(\epsilon) \right) \right\rVert \leq \frac{h_1}{\sqrt{n}} + \frac{h_2}{\sqrt{n}} \epsilon + h_3 \epsilon^2 . \qedhere
	\end{displaymath}
\end{proof}


Notice that~(\ref{eq:pnpr}) does not exactly match the Nonlinear PageRank problem at $p = 2 - \epsilon$ with $x = x (\epsilon)$, which has the form
	\begin{equation*}
		T B^\dag \left( \left( \left( B x(\epsilon) \right)^{\circ 2} + \zeta q \right)^{\circ - \frac{\epsilon}{2}} \odot B x(\epsilon) \right) = \beta r .
	\end{equation*}
	There is $B x^{0}$ in place of a $B x(\epsilon)$ term. However, this does not detract from the value of the Perturbed Nonlinear PageRank problem as a representation of the Nonlinear PageRank problem at $p = 2 - \epsilon$  with $x = x (\epsilon)$. $( ( B x^{0} )^{\circ 2} + \zeta q )^{\circ - \frac{\epsilon}{2}}$ as a stand-in for
	\begin{align*}
		\left( \left( B x(\epsilon) \right)^{\circ 2} + \zeta q \right)^{\circ - \frac{\epsilon}{2}} &= \left( \left( B x^{0} + \epsilon B x^{1} + \mathcal{O} (\epsilon^2) \right)^{\circ 2} + \zeta q \right)^{\circ - \frac{\epsilon}{2}} \\ 
		&= \left( \left( B x^{0} \right)^{\circ 2} + 2 \left( B x^{0} \odot \epsilon B x^{1} \right) + \mathcal{O} (\epsilon^2) + \zeta q \right)^{\circ - \frac{\epsilon}{2}}
	\end{align*}
	is a justified simplification when $\epsilon \ll 1$.
	
The expansion $x (\epsilon) = x^{0} + \epsilon x^{1} + \mathcal{O} (\epsilon^2)$ provides a means by which to gain an appreciation as to the impact of $p$ in the Nonlinear PageRank problem. This is captured in the following proposition. The conductance obtained from the solution of the Perturbed Nonlinear PageRank problem is compared to that of a reference set $\mathcal{F}$.


\begin{proposition}
	Let $\mathcal{G} = (\mathcal{V}, \mathcal{E}, w)$ with $\mathcal{E} \neq \emptyset$ and $n = | \mathcal{V} |$. Take $w (u, v) = 1$ for all $\{ u, v\} \in \mathcal{E}$. Every vertex $v \in \mathcal{V}$ has at most $\Delta$ neighbours. $\mathcal{F}$ is to be a proper subset of $\mathcal{V}$ containing the starting vertex $s$ such that
	\begin{itemize}
		\item $| \partial ( \mathcal{F} ) |$ is a maximum and
		\item no edge of $\partial ( \mathcal{F} )$ includes $s$.
	\end{itemize}
	
	$x^{0}$ is to be the least squares solution of the Nonlinear PageRank problem defined on $\mathcal{G}$ at $p = 2$. Assume that $( B x^{0} ) (v) + \zeta \leq 1$ for all $v \in \mathcal{V}$ and $0 \leq \epsilon \ll 1$. Consider the case where $n > h_1^2 / \epsilon^4$ and $( ( B x^{0} )^{\circ 2} + \zeta q )^{\circ - \frac{\epsilon}{2}}$ admits the decomposition
	\begin{equation*}
		\left( \left( B x^{0} \right)^{\circ 2} + \zeta q \right)^{\circ - \frac{\epsilon}{2}} = ( 1 + \sigma ) q + \delta , 
	\end{equation*}
	where
	\begin{equation*}
		\sigma = \frac{1}{|\mathcal{E}|} \sum_{i=1}^{|\mathcal{E}|} \left( \left( Bx^{0} \right)_i^2 + \zeta \right)^{-\frac{\epsilon}{2}} - 1
	\end{equation*} 
	and $\delta$ is the deviation.
	
	Denote the sweep set based on the Perturbed Nonlinear PageRank problem that returns the smallest conductance as $\mathcal{S}_{*}^{\epsilon} \subset \mathcal{V}$. The volume of subsets $\mathcal{F}$ and $\mathcal{S}_{*}^{\epsilon}$ are to be related by $\text{vol} \, ( \mathcal{F} ) = \tau \text{vol} \, ( \mathcal{S}_*^{\epsilon} )$, in which $\tau \in \mathbb{R}_{>0}$. Let $\zeta = \varepsilon \left( 1 + \sigma \right)^2 \tau$ with $\varepsilon \in \mathbb{R}_{> 0}$. For a sufficiently small $\lVert \delta \rVert$ along with $\text{vol}(\mathcal{F}) \leq \text{vol}(\mathcal{V} \setminus \mathcal{F})$ and $\text{vol} \, (\mathcal{S}_{*}^{\epsilon}) \leq \text{vol} \, (\mathcal{V} \setminus \mathcal{S}_{*}^{\epsilon})$,
	\begin{equation*}
			\Phi ( \mathcal{S}_{*}^{\epsilon} ) < \frac{1}{\varepsilon (1 + \sigma)^2} \Phi ( \mathcal{F} ) .
	\end{equation*}
\end{proposition}

\begin{proof}
	First of all, note that $\sigma$ represents a small non-negative value. Based on Proposition~\ref{prop:p2equiv}, $B x^0 = B c$. Since $T^{-1} \geqslant 0$, as stated in Proposition~\ref{prop:lprexists}, and $r(v) \geq 0$ for all $v \in \mathcal{V}$, $c (v) \geq 0$ for all $v \in \mathcal{V}$. Furthermore, $c (v) < 1$ for all $v \in \mathcal{V}$ according to Proposition~\ref{prop:max}, and thus the entries of $|B x^0|$ are less than $1$. The assumption that $( B x^{0} ) (v) + \zeta \leq 1$ for all $v \in \mathcal{V}$ together with $\epsilon \ll 1$ facilitates a small non-negative $\sigma$.
	
	For $\delta = 0$, the Perturbed Nonlinear PageRank problem~(\ref{eq:pnpr}) has the form 
	\begin{displaymath}
		(1 + \sigma) T B^\dag B y (\epsilon) = \beta r .
	\end{displaymath}
	The least squares solution $(1 + \sigma) y (\epsilon) = c - (1/n) q$ follows from Proposition~\ref{prop:p2equiv} and
	\begin{displaymath}
		y (\epsilon) = \frac{1}{1 + \sigma} c - \frac{1}{n (1 + \sigma)} q .
	\end{displaymath}
	
	Let $M = \text{diag} \, (\delta_1, \ldots, \delta_{|\mathcal{E}|})$ and $N = (1 + \sigma) B^\dag B + B^\dag M B $. Then, the Perturbed Nonlinear PageRank problem can be written as $T N x (\epsilon) = \beta r$. Suppose that $\kappa q + z \in \text{ker} \, (N)$ with $q^\top z = 0$, where $\kappa \in \mathbb{R}$ and $z \in \mathbb{R}^n$, is an orthogonal decomposition. However,
	\begin{displaymath}
		0 = N (\kappa q + z) = \kappa N q + Nz = \kappa  (1 + \sigma) B^\dag B q + \kappa B^\dag M B q + N z = Nz ,
	\end{displaymath}
	as $B q = 0$. Since $B^\dag B = I - (1/n) q q^\top$ from Proposition~\ref{prop:mp} and $q^\top z = 0$, it follows that
	\begin{displaymath}
		0 = N z = (1 + \sigma) z + B^\dag M B z .
	\end{displaymath}
	Hence, 
	\begin{displaymath}
		\lVert z \rVert = \frac{1}{1 + \sigma} \left\lVert B^\dag M B z \right\rVert .
	\end{displaymath}
	In view of $\lVert M \rVert = \max_{1 \leq i \leq | \mathcal{E} |} | \delta_i | \leq \lVert \delta \rVert$,
	\begin{displaymath}
		\lVert z \rVert \leq \frac{1}{1 + \sigma} \lVert B^\dag \rVert  \lVert \delta \rVert  \lVert B \rVert  \lVert z \rVert .
	\end{displaymath}
	A sufficiently small $\lVert \delta \rVert$, i.e.,
	 \begin{displaymath}
	 	\frac{1}{1 + \sigma} \lVert B^\dag \rVert  \lVert \delta \rVert  \lVert B \rVert < 1 ,
	 \end{displaymath}
	 implies $\lVert z \rVert = 0$. Thus, $\text{ker} \, (N) = \text{span} \, (q)$ and so $\text{rank} (N) = n - 1$. 
	 
	 Observe that $\text{rank} \, (N) = \text{rank} \, (B^\dag B)$. This is a consequence of $\text{ker} \, (B^\dag B) = \text{ker} \, (B)$ and Proposition~\ref{prop:rnkB}. If $z \in \text{ker} \, (B^\dag B)$, then $B^\dag B z = 0$ and $B z \in \text{ker} \, (B^\dag)$. But $\text{ker} \, (B^\dag) = \text{ker} \, (B^\top)$, so $B^\top B z = 0$. Consider $z^\top B^\top B z = \lVert B z \rVert^2 = 0$. Thus, $B z = 0$ and $z \in \text{ker} \, (B)$. Conversely, if $z \in \text{ker} \, (B)$, then $B z = 0$. This results in $B^\dag B z = 0$ and $z \in \text{ker} \, (B^\dag B)$.
	 
	 The bounded difference between $x (\epsilon)$ and $y (\epsilon)$ is
	 \begin{displaymath}
	 	\lVert x (\epsilon) - y (\epsilon) \rVert = \left\lVert N^\dag c -  \left( (1 + \sigma) B^\dag B \right)^\dag c \right\rVert \leq \left\lVert N^\dag -  \left( (1+\sigma) B^\dag B \right)^\dag \right\rVert  \lVert c \rVert .
	 \end{displaymath}
	 Because $\lVert \delta \rVert$ is small, $N$ can be treated as a perturbed matrix. Then, by~\cite[Theorem~3.4 on p.~645]{stewart1977perturb},
	 \begin{align*}
	 	\left\lVert N^\dag - \left( (1+\sigma) B^\dag B \right)^\dag \right\rVert &\leq \frac{1 + \sqrt{5}}{2}  \left\lVert \left( (1 + \sigma) B^\dag B \right)^\dag \right\rVert  \left\lVert N^\dag \right\rVert  \left\lVert B^\dag M B \right\rVert \\
	 	&= \frac{1 + \sqrt{5}}{2 (1 + \sigma)}  \left\lVert B^\dag B \right\rVert  \left\lVert N^\dag \right\rVert  \left\lVert B^\dag M B \right\rVert  ,
	 \end{align*}
	 where the resulting equality makes use of Proposition~\ref{prop:mpi}. For any $\{ u, v\} \in \mathcal{E}$,
	 \begin{align*}
	 	\left| x^\epsilon (u) - x^\epsilon (v) - \left( y^\epsilon (u) - y^\epsilon (v) \right) \right|  &\leq  \left| x^\epsilon (u) - y^\epsilon (u) \right| + \left| x^\epsilon (v) - y^\epsilon (v) \right| \\  
	 	&\leq 2 \lVert x (\epsilon) - y (\epsilon) \rVert \\
	 	&\leq \frac{1 + \sqrt{5}}{1 + \sigma} \left\lVert B^\dag B \right\rVert  \left\lVert N^\dag \right\rVert  \left\lVert B^\dag \right\rVert  \lVert \delta \rVert  \lVert B \rVert \lVert c \rVert \\
	 	&\leq \frac{1 + \sqrt{5}}{1 + \sigma} \left\lVert N^\dag \right\rVert  \left\lVert B^\dag \right\rVert^2 \lVert B \rVert^2 \lVert \delta \rVert \lVert c \rVert ,
	 \end{align*}
	 in which $x^\epsilon = x (\epsilon)$ and $y^\epsilon = y (\epsilon)$. Subsequently,
	 \begin{displaymath}
	 	\left| x^\epsilon (u) - x^\epsilon (v) - \frac{c(u) - c(v)}{1 + \sigma} \right|  =  \mathcal{O} ( \lVert \delta \rVert ) .
	 \end{displaymath}
	 
	 The cornerstone of the concluding argument rests upon $\sum_{ \{u , v \} \in \partial (\mathcal{S}_{*}^{\epsilon}) } ( ( c(u) - c(v) )^2 + \zeta )$. An assumption of $(Bx^0) (v) + \zeta \leq 1$ for all $v \in \mathcal{V}$ supports
	\begin{displaymath}
		\sum_{ \{u , v \} \in \partial (\mathcal{S}_{*}^{\epsilon}) } \left( \left( (1 + \sigma) \left(x^{\epsilon} (u) - x^{\epsilon} (v) \right) + \mathcal{O} \left( \lVert \delta \rVert \right) \right)^2 + \zeta \right) < | \partial ( \mathcal{F} ) | .
	\end{displaymath}
	It follows that
		\begin{displaymath}
		\sum_{ \{u , v \} \in \partial (\mathcal{S}_{*}^{\epsilon}) } \left( \left( x^{\epsilon} (u) - x^{\epsilon} (v) \right)^2 + \mathcal{O} \left( \lVert \delta \rVert \right) + \frac{\zeta}{\left( 1 + \sigma \right)^2} \right) < \frac{1}{\left( 1 + \sigma \right)^2} | \partial ( \mathcal{F} ) | .
	\end{displaymath}
	As $x^\epsilon (u) - x^\epsilon (v) = 0$ cannot be discounted for any $\{u , v \} \in \partial (\mathcal{S}_{*}^{\epsilon})$,
	\begin{displaymath}
		\sum_{ \{u , v \} \in \partial (\mathcal{S}_{*}^{\epsilon}) } \left( \left( x^{\epsilon} (u) - x^{\epsilon} (v) \right)^2 + \mathcal{O} \left( \lVert \delta \rVert \right) + \frac{\zeta}{\left( 1 + \sigma \right)^2} \right) \geq \varepsilon \tau | \partial ( \mathcal{S}_{*}^{\epsilon} ) | + \mathcal{O} \left( \lVert \delta \rVert \right) ,
	\end{displaymath}
	This leads to
	\begin{displaymath}
		\varepsilon \tau | \partial ( \mathcal{S}_{*}^{\epsilon} ) | < \frac{1}{(1 + \sigma)^2} | \partial ( \mathcal{F} ) | \qquad \text{and} \qquad \Phi ( \mathcal{S}_{*}^{\epsilon} ) < \frac{1}{\varepsilon (1 + \sigma)^2} \Phi ( \mathcal{F} ) . \qedhere
	\end{displaymath}
\end{proof}

\section{Numerical methodology}
\label{sec:numerics}
Consider $\mathcal{G} = (\mathcal{V}, \mathcal{E}, w)$ with $\mathcal{E} \neq \emptyset$. Again, $n = | \mathcal{V} |$ and $m = | \mathcal{E} |$. Let $g(x) = \beta r - T f(x)$. For the purpose of obtaining a numerical solution, the Nonlinear PageRank problem is re-expressed as follows: Given $g: \mathbb{R}^n \to \mathbb{R}^n$, find $x \in \mathbb{R}^n$ that satisfies
\begin{equation}\label{eq:nse}
	g (x) = 0 .
\end{equation}
A solution or root of the system of nonlinear equations is denoted by $x^*$. If a real-valued function has the form
\begin{equation*}
	\psi (x) = \frac{1}{2} g(x)^\top g(x) = \frac{1}{2} \left\lVert g(x) \right\rVert_2^2 ,
\end{equation*}
then $\psi ( x^* ) = 0$. Clearly, $x^*$ is at least a local minimiser of $\psi$, that is,
\begin{equation*}
	\psi (x^*) \leq \psi (x) \text{ for all $x$ near $x^*$} ,
\end{equation*}
as $\psi ( x ) \geq 0$ for all $x \in \mathbb{R}^n$. It follows that the optimisation problem
\begin{equation}\label{eq:min}
	\min_{x \in \mathbb{R}^n} \psi (x)
\end{equation}
solves for $x^*$. $\psi (x)$ is referred to as a (least squares) merit function.

\begin{remark}
	Even though a local minimiser of $\psi$ need not be a solution of~(\ref{eq:nse}), the merit function $\psi$ has nonetheless been used successfully in practice~\cite[Chapter~11.2]{nocedal1999numerical}.
\end{remark}

An approach involving the Levenberg-Marquardt method with a full rank Jacobian variant applied to problem~(\ref{eq:min}) is now presented.

\subsection{Levenberg-Marquardt method}
Optimisation algorithms proceed from an initial iterate $x_0$ and generate a finite sequence of iterates $\{ x_k \}_{k > 0}$ that either converge towards a sufficiently accurate approximation of the solution $x^*$, or represent the failed attempt to do so. Algorithms are distinguished by the transition from a current iterate $x_c$ to the next iterate $x_+$.

In the Levenberg-Marquardt method, a local quadratic model 
\begin{equation*}
	\theta_c (x) = \psi (x_c) + (x - x_c)^\top \nabla \psi (x_c)  + \frac{1}{2} (x - x_c)^\top H_c (x - x_c) ,
\end{equation*}
is constructed at $x_c$. Here, the gradient of $\psi$ equates to
\begin{equation*}
	\nabla \psi (x_c) = J (x_c)^\top g(x_c) ,
\end{equation*}
and the model Hessian is given by
\begin{equation*}
	H_c = H (x_c) = J (x_c)^\top J (x_c) + \lambda_c I ,
\end{equation*}
in which $\lambda_c \in \mathbb{R}_{\geq 0}$ acts as a regularisation parameter. Both the gradient and Hessian feature the $n \times n$ Jacobian matrix $J = \partial g / \partial x$.


\begin{proposition}\label{prop:jac}
	Let $\beta \in (0,1)$ and $p \in (1,2]$. For $g(x) = \beta r - T f(x)$ and $f(x) = B^\dag \left( \left( (Bx)^{\circ 2} + \zeta q \right)^{\circ \frac{1}{2} ( p - 2 )} \odot Bx \right)$, the Jacobian has the form
		\begin{equation*}
			J = - T B^\dag K B ,
		\end{equation*}
	where
		\begin{multline*}
			K = \text{diag} \left( \left( z_1^2 + \zeta \right)^{\frac{1}{2} (p-2)} + (p-2) z_1^2 \left( z_1^2 + \zeta \right)^{\frac{1}{2} (p-4)},  \ldots, \left( z_m^2 + \zeta \right)^{\frac{1}{2} (p-2)} \right. \\ \left. + (p-2) z_m^2 \left( z_m^2 + \zeta \right)^{\frac{1}{2} (p-4)} \right)
		\end{multline*}
	and $z = Bx$.
\end{proposition}

\begin{proof}
	By definition, 
	\begin{displaymath}
		J = \frac{\partial g}{ \partial x} = - T \frac{\partial }{\partial x} f(x) .
	\end{displaymath}
	Let $z = Bx$. Then
	\begin{align*}
		\frac{\partial f_k}{\partial x_j} &= \frac{\partial f_k}{\partial z_i} \frac{\partial z_i}{\partial x_j} \\
		&= b_{ki}^\dag \left( \left( z_i^2 + \zeta \right)^{\frac{1}{2} (p-2)} + (p - 2) z_i^2 \left( z_i^2 + \zeta \right)^{\frac{1}{2} (p-4)} \right) b_{ij}
	\end{align*}
	for $j, k = 1, 2, \ldots, n$. Summation is implied over the repeated index.
\end{proof}


\begin{corollary}
	The Jacobian $J$ is (locally) Lipschitz continuous.
\end{corollary}

\begin{proof}
	Define $a = y + \mu (x - y)$, in which $y \in \mathbb{R}^n$ and $\mu \in [0,1]$. Then
	\begin{displaymath}
		\frac{d}{d \mu} J (y + \mu (x - y)) = \frac{\partial J}{\partial a} \frac{d a}{d \mu} = J ' (a(\mu)) \bar{\times}_3 (x-y) .
	\end{displaymath}
	By the fundamental theorem of calculus,
	\begin{displaymath}
		J (x) - J (y) = \int_0^1 \, J ' (a(\mu)) \bar{\times}_3 (x-y) \, d \mu .
	\end{displaymath}
	Let $\mathtt{\mathbf{J}}$ be a matricization of $J ' (a(\mu))$. Note that
	\begin{displaymath}
		\lVert J ' (a(\mu)) \bar{\times}_3 (x-y) \rVert \leq \lVert J ' (a(\mu)) \bar{\times}_3 (x-y) \rVert_\text{F} = \lVert \mathtt{\mathbf{J}} (x-y) \rVert
	\end{displaymath}
	and $\lVert \mathtt{\mathbf{J}} (x-y) \rVert \leq \lVert \mathtt{\mathbf{J}} \rVert_\text{F} \lVert x-y \rVert = \lVert J ' (a(\mu)) \rVert_\text{F} \lVert x-y \rVert$. Subsequently,
	\begin{align*}
		\left\lVert J (x) - J (y) \right\rVert &= \left\lVert \int_0^1 \, J ' (a(\mu)) \bar{\times}_3 (x-y) \, d \mu \right\rVert \\
		&\leq \int_0^1 \, \left\lVert J ' (a(\mu)) \right\rVert_\text{F} \left\lVert x-y \right\rVert \, d \mu \\
		&\leq \xi \left\lVert x - y \right\rVert ,
	\end{align*}		
	where $\xi = \max_{\mu \in [0,1]} \lVert J ' (a(\mu)) \rVert_\text{F} < \infty$. This satisfies the Lipschitz condition.
\end{proof}


\begin{corollary}\label{cor:rnkjac}
	$\text{ker} \, (J) = \text{ker} \, ( B )$, and $\text{rank} \, (J) = n - 1$.
\end{corollary}

\begin{proof}
	First, observe that $K$ is nonsingular, as the diagonal entries of $K$ are all non-zero and positive. Each diagonal entry has the form
	\begin{displaymath}
		\left( z_i^2 + \zeta \right)^{\frac{1}{2} (p - 2)} + (p - 2) z_i^2 \left( z_i^2 + \zeta \right)^{\frac{1}{2} (p - 4)}
	\end{displaymath}
	for $i = 1, \ldots, m$. Let $\eta = \left( z_i^2 + \zeta \right)^{\frac{1}{2}}$ for any $i \in \{ 1, \ldots, m \}$. Since $\eta > 0$ and $p \in (1 , 2]$, it follows that
	\begin{displaymath}
		\eta^{p - 2} + (p - 2) ( \eta^2 - \zeta) \eta^{p - 4} = \eta^{p - 2} + (p - 2) \eta^{p - 2} - (p - 2) \zeta \eta^{p - 4}  = (p - 1) \eta^{p - 2} - (p - 2) \zeta \eta^{p - 4} > 0 .
	\end{displaymath}
	
	Next, $\text{ker} \, ( B ) = \text{ker} \, ( T B^\dag K B )$ is tackled. If $x \in \text{ker} \, (B)$, then $B x = 0$. Thus, $T B^\dag K B x = 0$, and $x \in \text{ker} \, ( T B^\dag K B )$. Conversely, if $x \in \text{ker} \, ( T B^\dag K B )$, then $T B^\dag K B x = 0$. But $T^{-1} T B^\dag K B x = B^\dag K B x = 0$. Since $\text{ker} \, ( B^\dag ) = \text{ker} \, ( B^\top )$, it suffices to consider $B^\top K B x = 0$. Take $\kappa$ to be the smallest of all the (positive) diagonal entries of $K$. Subsequently, $\kappa (B x)^\top (B x) \leq (B x)^\top K (B x) = 0$. Thus, $B x = 0$, and $x \in \text{ker} \, (B)$. By Proposition~\ref{prop:rnkB}, the Jacobian is of rank $n - 1$.
\end{proof}


\begin{proposition}\label{prop:hess}
	The Hessian $H$ is (i) symmetric and (ii) positive definite when $\lambda > 0$.
\end{proposition}

\begin{proof}$ $\hfill
	\begin{enumerate}[label=(\roman*)]
		\item $H^T = \left( J^\top J + \lambda I \right)^\top = J^\top J + \lambda I = H$.
		\item Given $y \in \mathbb{R}^n \setminus{ \{ 0 \} }$,
		\begin{displaymath}
			y^\top J^\top J y = \left( J y \right)^\top  J y \geq 0 \text{ and } \lambda \left( y^\top I y \right) > 0.
		\end{displaymath}
		Thus, $y^\top H y > 0$. \qedhere
	\end{enumerate}
\end{proof}


A necessary condition (refer to~\cite[Theorem~2.2 on p.~15]{nocedal1999numerical} and~\cite[Theorem~1.3.1 on p.~5]{kelley1999optimization}) for a local minimiser $x_t$ of $\theta_c$ is 
\begin{equation*}
	0 = \nabla \theta_c (x_t) = \nabla \psi (x_c) + H_c (x_t - x_c) .
\end{equation*}
For $\lambda_c > 0$, it follows from Proposition~\ref{prop:hess} (and the nonsingular nature of a positive definite matrix) that the minimiser is the unique solution
\begin{equation}\label{eq:trialsol}
	x_t = x_c - \left( J (x_c)^\top J (x_c) + \lambda_c I  \right)^{-1} J (x_c)^\top g(x_c) .
\end{equation}
$x_t$ is treated as a trial solution that could be the next iterate $x_+$ depending on how well the quadratic model approximates $\psi$. This is assessed through a comparison of the actual reduction
\begin{equation*}
	ared = \psi (x_c) - \psi (x_t) 
\end{equation*}
with the predicted reduction
\begin{align*}
	pred &= \theta_c (x_c) - \theta_c(x_t) \\
	&= - (x_t - x_c)^\top J (x_c)^\top g(x_c) - \frac{1}{2} (x_t - x_c)^\top \left( J (x_c)^\top J (x_c) + \lambda_c I  \right) (x_t - x_c) \\ 
	&= - (x_t - x_c)^\top J (x_c)^\top g(x_c) + \frac{1}{2} (x_t - x_c)^\top J (x_c)^\top g(x_c) = - \frac{1}{2} (x_t - x_c)^\top \nabla \psi (x_c) ,
\end{align*}
which is captured in the ratio
\begin{equation*}
	\varrho = \frac{ared}{pred} = - 2 \frac{\psi (x_c) - \psi (x_t)}{(x_t - x_c)^\top \nabla \psi (x_c)} .
\end{equation*}
If $\varrho = 1$, then the quadratic model faithfully reproduces the behaviour of $\psi (x)$ around $x_c$. The trial solution is accepted for a value of $\varrho $ sufficiently greater than zero, and $\lambda_c$ is possibly decreased. Otherwise, $\lambda_c$ is increased and a new trial solution is computed. An algorithm detailing the steps in each iteration of the Levenberg-Marquardt method can be found in~\cite[Algorithm~3.3.5 on p.~58]{kelley1999optimization}. Quadratic convergence can be achieved~\cite[Theorem~3.3.4 on p.~58]{kelley1999optimization}. However, this is contingent on the sequence $\{ x_k \}$ having the limit $x^*$ and $J (x^*)$ being full rank.

\begin{remark}
	Notice that the form of $T$ has no bearing on the preceding analysis. It is thus possible to adapt $T$ without impacting the methodology as long as $T$  remains nonsingular and constant, e.g. $T = \beta I + D^{-1} L$.
\end{remark}

\subsection{Rank deficiency and the Jacobian}
The issue of a rank-deficient Jacobian in nonlinear least squares problems was examined in~\cite{ipsen2011rank} but for non-zero (though small) $\psi (x^*)$. Subset selection applied to the Jacobian was recommended over a truncated singular value decomposition. This involves the formation of a full rank variant, say $\widetilde{J}$, comprised solely of linearly independent columns from $J$. Instead of relying on~(\ref{eq:trialsol}), the system of linear equations 
\begin{equation*}
 	\left( \widetilde{J} (x_c)^\top \widetilde{J} (x_c) + \lambda_c I  \right) \left( \widetilde{x}_t - \widetilde{x}_c \right) = - \widetilde{J} (x_c)^\top g(x_c)
\end{equation*}
can be solved for $\left( \widetilde{x}_t - \widetilde{x}_c \right)$. Entries of $x$ corresponding to the columns of $J$ not present in $\widetilde{J}$ would be fixed to nominal values, and not included in $\widetilde{x}_c$ and $\widetilde{x}_t$. Let $\widetilde{H}_c = \widetilde{J} (x_c)^\top \widetilde{J} (x_c) + \lambda_c I$. Recovery of the solution in the standard methodology most suited to a dense and not so large (reduced) Hessian matrix starts with a Cholesky factorisation of $\widetilde{H}_c$, and is followed by two triangular system solves.

There is the outstanding question though of how to discover the linearly independent columns of the Jacobian in the first place. A rank revealing algorithm could be employed for this task, but there is no need. After all, the underlying graph-based nature of the Nonlinear PageRank problem can be leveraged to determine the column in the Jacobian (see Corollary~\ref{cor:rnkjac}) that will play no part in the full rank variant. Since the problem solution figures in the graph partitioning process, it seems most prudent to fix the value of $x (v)$ for $v \in \mathcal{V}$ that is the \emph{furthest (minimum weighted) distance} from the starting vertex $s \in \mathcal{V}$. Clearly, the relevant column in the Jacobian would then be the one omitted from $\widetilde{J}$.

\begin{remark}
	An alternative strategy could be to exclude the column in the Jacobian corresponding to the smallest entry of the minimal least squares solution to the Nonlinear PageRank problem at $p=2$ (see Proposition~\ref{prop:p2equiv}).
\end{remark}

\section{Experiments}
\label{sec:exp}
Once a solution to the Nonlinear PageRank problem is available, graph partitioning can take place as outlined in Section~\ref{sec:conductance}. Results pertaining to the local cluster quality on both synthetic and real-world inspired graphs are now reported.

\subsection{Set-up}
\label{sec:algo}

An implementation consisting of the Nonlinear PageRank problem numerical solution and sweep-cut was made in MATLAB R2024b. \texttt{immoptibox}~\cite{immoptibox} was adopted for the Levenberg-Marquardt method with adjustments to accommodate the full rank Jacobian variant. Dijkstra’s algorithm~\cite{dijkstra1959note} was called on for determining the furthest distance vertex from the starting vertex. The source code can be found at \url{https://github.com/DmsPas/Nonlinear_modified_PageRank}.

For the problem solution, $\beta = 0.01$ unless specified differently. The value at the vertex judged to be the furthest was $10^{-12}$. $\zeta$ was either $10^{-11}$ for graphs with fewer than $10^{4}$ vertices or $10^{-6}$ otherwise. Termination criteria in the form of a gradient (maximum) norm and a relative change in the potential solution were set to $10^{-7}$.

In order to determine the local cluster that has the smallest conductance with respect to $p$, the Nonlinear PageRank problem defined on a graph $\mathcal{G}$ was tackled for $p = 1.95$, $1.9$, $1.8$, $1.7$, $1.6$, $1.5$, and $1.45$ sequentially (these values were selected based on experiment). Each solution served as the initial iterate $x_0$ for the next optimisation problem with the exception of $p = 1.95$, which was the minimal least squares solution to the Nonlinear PageRank problem at $p = 2$. $\lambda_0$ was taken to be the product of $10^{-3}$ and the largest entry in the main diagonal of $\widetilde{J} (x_0)^\top \widetilde{J} (x_0)$.

Local clusters based on the Nonlinear PageRank (hereafter abbreviated as NPR) problem solution were compared to those from sweep-cuts of solutions resulting from the following:
\begin{itemize}
	\item \textbf{Approximate Personalised PageRank (APPR)}~\cite{gleich2006appr}. An efficient algorithm for computing accurate approximations of the PageRank problem~(\ref{eq:eigpr}).
	\item \textbf{Nonlinear Power Diffusion (NPD)}~\cite{ibrahim2019nonlinear}. A graph diffusion model, in which the Laplacian matrix acts on an (element-wise) power function, is solved by way of the Euler method that is derived from the forward finite difference approximation of the time derivative.
	\item \textbf{$p$-Laplacian Diffusion ($p$-DIFF)}~\cite{ibrahim2019nonlinear}. A graph diffusion model, in which the $p$-Laplacian takes the place of the Laplacian matrix term, is solved in exactly the same manner as NPD.
\end{itemize}
Source code for the competing methodologies is in the public domain. Operating parameter settings were maintained in the experiments. For all graphs, a starting vertex was selected at random, unless otherwise specified, and the conductance of the local cluster formed at each $p$ was ascertained. This was repeated a total of 50 times for every synthetic graph, and 10 times for the real-world inspired graphs.

Vertices in every graph were assigned a binary ground truth label to identify those that constitute a local cluster. The labels supported quality assessment of the discovered clusters through a measure called (balanced) $Fscore \in [0,1]$, which is defined as
\begin{equation*}
    \label{eq:fscore}
    Fscore = 2 \, \frac{precision \cdot recall}{ precision + recall},
\end{equation*}
where $precision = TP/(TP + FP)$ and $recall = TP/(TP + FN)$. $TP$ represents the number of true positives, i.e., vertices correctly classified in the local cluster. $FP$ is the number of false positives, i.e., vertices incorrectly classified in the local cluster. $FN$ stands for the number of false negatives, i.e., vertices that should have been part of the local cluster but were not. A value of $Fscore = 1$ indicates perfect recovery of a local cluster. Mean and standard deviation of both the conductance and related $Fscore$ were recorded as standard, that is, where relevant.

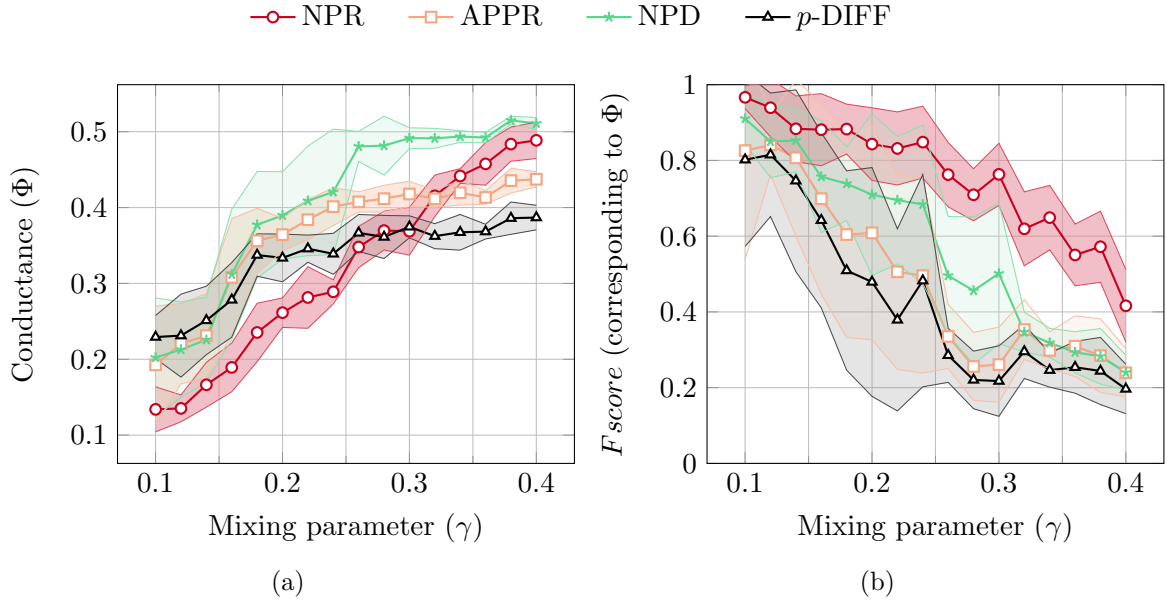
\begin{figure}[t!]
	\centering
	\begin{minipage}{\textwidth}
	\centering
	\begin{tikzpicture}
  \centering
  \begin{customlegend}[
      legend columns=4,
      legend style={
      anchor=north,
      draw=none,
      /tikz/every even column/.append style={column sep=0.5cm}},
      legend entries={NPR\\ APPR\\ NPD\\ $p$-DIFF\\},
      legend image post style={xscale=1.0},]
    \addlegendimage{mycolor1, mark=*, mark options={fill=white}, thick}
    \addlegendimage{mycolor2, mark=square*, mark options={fill=white}, thick}
    \addlegendimage{mycolor3, mark=star, mark options={fill=white}, thick}
    \addlegendimage{black, mark=triangle*, mark options={fill=white}, thick}
    \addlegendimage{mycolor5, mark=diamond*, mark options={fill=white}, thick}
  \end{customlegend}
\end{tikzpicture}
    \vspace{0.8em}
	\end{minipage}
	\subcaptionbox{\label{fig:SBM_CCut}}
    {\begin{tikzpicture}[y=.2cm, x=.7cm]
  \centering
  \begin{axis}[
    width=0.49\columnwidth,
    grid=both,
    tick align=inside,
    yticklabel style={
        /pgf/number format/fixed,
        /pgf/number format/precision=5
},
scaled y ticks=false,
    xlabel={Mixing parameter ($\gamma$)},
    ylabel={Conductance ($\Phi$)},
    minor x tick num = 1,
    legend pos=north west,
  ]
    \addplot[mycolor1, mark=*, mark options={fill=white}, thick] table[x=k, y=NPR, header=true, col sep=comma] {SBM_CCut.dat};
    \addplot[mycolor2, mark=square*, mark options={fill=white}, thick] table[x=k, y=APPR, header=true, col sep=comma] {SBM_CCut.dat};
    \addplot[mycolor3, mark=star, mark options={fill=white}, thick] table[x=k, y=NPD, header=true, col sep=comma] {SBM_CCut.dat};
    \addplot[black, mark=triangle*, mark options={fill=white}, thick] table[x=k, y=p-Lap, header=true, col sep=comma] {SBM_CCut.dat};        
    \addplot[name path=NPR_top,color=mycolor1!70] coordinates {
    (0.1, 0.133972+ 0.029633)
    (0.12,0.135256+ 0.017696)
    (0.14,0.166374+ 0.029544)
    (0.16,0.189264+ 0.032068)
    (0.18,0.235340+ 0.038517)
    (0.20,0.261457+ 0.019420)
    (0.22,0.281457+ 0.040615)
    (0.24,0.288936+ 0.015802)
    (0.26,0.347969+ 0.027248)
    (0.28,0.369648+ 0.025528)
    (0.30,0.368829+ 0.031566)
    (0.32,0.415485+ 0.027431)
    (0.34,0.441604+ 0.009699)
    (0.36,0.457593+ 0.028094)
    (0.38,0.483661+ 0.022618)
    (0.40,0.488646+ 0.023997)};
    \addplot[name path=NPR_bottom,color=mycolor1!70] coordinates {
    (0.1, 0.133972- 0.029633)
    (0.12,0.135256- 0.017696)
    (0.14,0.166374- 0.029544)
    (0.16,0.189264- 0.032068)
    (0.18,0.235340- 0.038517)
    (0.20,0.261457- 0.019420)
    (0.22,0.281457- 0.040615)
    (0.24,0.288936- 0.015802)
    (0.26,0.347969- 0.027248)
    (0.28,0.369648- 0.025528)
    (0.30,0.368829- 0.031566)
    (0.32,0.415485- 0.027431)
    (0.34,0.441604- 0.009699)
    (0.36,0.457593- 0.028094)
    (0.38,0.483661- 0.022618)
    (0.40,0.488646- 0.023997)};
    \addplot[mycolor1!50,fill opacity=0.5] 
    fill between[of=NPR_top and NPR_bottom];
    \addplot[name path=APPR_top,color=mycolor2!70] coordinates {
    (0.1, 0.1926+0.0776)
    (0.12,0.2206+0.0529)
    (0.14,0.2311+0.0555)
    (0.16,0.3080+0.0772)
    (0.18,0.3562+0.0428)
    (0.20,0.3643+0.0216)
    (0.22,0.3840+0.0265)
    (0.24,0.4012+0.0249)
    (0.26,0.4075+0.0133)
    (0.28,0.4117+0.0177)
    (0.30,0.4179+0.0168)
    (0.32,0.4116+0.0112)
    (0.34,0.4195+0.0162)
    (0.36,0.4133+0.0089)
    (0.38,0.4355+0.0164)
    (0.40,0.4373+0.0100)};
    \addplot[name path=APPR_bottom,color=mycolor2!70] coordinates {
    (0.1, 0.1926-0.0776)
    (0.12,0.2206-0.0529)
    (0.14,0.2311-0.0555)
    (0.16,0.3080-0.0772)
    (0.18,0.3562-0.0428)
    (0.20,0.3643-0.0216)
    (0.22,0.3840-0.0265)
    (0.24,0.4012-0.0249)
    (0.26,0.4075-0.0133)
    (0.28,0.4117-0.0177)
    (0.30,0.4179-0.0168)
    (0.32,0.4116-0.0112)
    (0.34,0.4195-0.0162)
    (0.36,0.4133-0.0089)
    (0.38,0.4355-0.0164)
    (0.40,0.4373-0.0100)};
    \addplot[mycolor2!50,fill opacity=0.5] 
    fill between[of=APPR_top and APPR_bottom];
    \addplot[name path=NPD_top,color=mycolor3!70] coordinates {
    (0.1, 0.2023+0.0786)
    (0.12,0.2128+0.0627)
    (0.14,0.2257+0.0560)
    (0.16,0.3119+0.0860)
    (0.18,0.3773+0.0702)
    (0.20,0.3897+0.0577)
    (0.22,0.4089+0.0722)
    (0.24,0.4206+0.0828)
    (0.26,0.4806+0.0198)
    (0.28,0.4816+0.0388)
    (0.30,0.4915+0.0137)
    (0.32,0.4913+0.0131)
    (0.34,0.4936+0.0076)
    (0.36,0.4923+0.0069)
    (0.38,0.5149+0.0054)
    (0.40,0.5110+0.0078)};
    \addplot[name path=NPD_bottom,color=mycolor3!70] coordinates {
    (0.1, 0.2023-0.0786)
    (0.12,0.2128-0.0627)
    (0.14,0.2257-0.0560)
    (0.16,0.3119-0.0860)
    (0.18,0.3773-0.0702)
    (0.20,0.3897-0.0577)
    (0.22,0.4089-0.0722)
    (0.24,0.4206-0.0828)
    (0.26,0.4806-0.0198)
    (0.28,0.4816-0.0388)
    (0.30,0.4915-0.0137)
    (0.32,0.4913-0.0131)
    (0.34,0.4936-0.0076)
    (0.36,0.4923-0.0069)
    (0.38,0.5149-0.0054)
    (0.40,0.5110-0.0078)};
    \addplot[mycolor3!50,fill opacity=0.2] 
    fill between[of=NPD_top and NPD_bottom];
    \addplot[name path=pLap_top,color=black!70] coordinates {
    (0.1, 0.2293+0.0284)
    (0.12,0.2312+0.0547)
    (0.14,0.2512+0.0451)
    (0.16,0.2783+0.0492)
    (0.18,0.3374+0.0279)
    (0.20,0.3336+0.0315)
    (0.22,0.3458+0.0179)
    (0.24,0.3390+0.0270)
    (0.26,0.3666+0.0243)
    (0.28,0.3612+0.0284)
    (0.30,0.3746+0.0145)
    (0.32,0.3623+0.0165)
    (0.34,0.3674+0.0235)
    (0.36,0.3683+0.0097)
    (0.38,0.3860+0.0213)
    (0.40,0.3868+0.0162)};
    \addplot[name path=pLap_bottom,color=black!70] coordinates {
    (0.1, 0.2293-0.0284)
    (0.12,0.2312-0.0547)
    (0.14,0.2512-0.0451)
    (0.16,0.2783-0.0492)
    (0.18,0.3374-0.0279)
    (0.20,0.3336-0.0315)
    (0.22,0.3458-0.0179)
    (0.24,0.3390-0.0270)
    (0.26,0.3666-0.0243)
    (0.28,0.3612-0.0284)
    (0.30,0.3746-0.0145)
    (0.32,0.3623-0.0165)
    (0.34,0.3674-0.0235)
    (0.36,0.3683-0.0097)
    (0.38,0.3860-0.0213)
    (0.40,0.3868-0.0162)};
    \addplot[black!50,fill opacity=0.2] 
    fill between[of=pLap_top and pLap_bottom];
  \end{axis}
  
\end{tikzpicture}}\hspace{0.2em}\hfill%
	\subcaptionbox{\label{fig:SBM_Fscore}}
    {\begin{tikzpicture}[y=.2cm, x=.7cm]
  \centering
  \begin{axis}[
    width=0.49\columnwidth,
    grid=both,
    tick align=inside,
    yticklabel style={
        /pgf/number format/fixed,
        /pgf/number format/precision=5
},
scaled y ticks=false,
    xlabel={Mixing parameter ($\gamma$)},
    ylabel={$Fscore$ (corresponding to $\Phi$)},
    minor x tick num = 1,
    ymax=1,
    ymin=0.0,
    legend pos=north west,
  ]
    \addplot[mycolor1, mark=*, mark options={fill=white}, thick] table[x=k, y=NPR, header=true, col sep=comma] {SBM_Fscore.dat};
    \addplot[mycolor2, mark=square*, mark options={fill=white}, thick] table[x=k, y=APPR, header=true, col sep=comma] {SBM_Fscore.dat};
    \addplot[mycolor3, mark=star, mark options={fill=white}, thick] table[x=k, y=NPD, header=true, col sep=comma] {SBM_Fscore.dat};
    \addplot[black, mark=triangle*, mark options={fill=white}, thick] table[x=k, y=p-Lap, header=true, col sep=comma] {SBM_Fscore.dat};        
    \addplot[name path=NPR_top,color=mycolor1!70] coordinates {
    (0.1, 0.966636+ 0.029638)
    (0.12,0.939154+ 0.076717)
    (0.14,0.883484+ 0.086820)
    (0.16,0.881140+ 0.095176)
    (0.18,0.882741+ 0.065913)
    (0.20,0.843096+ 0.095891)
    (0.22,0.831645+ 0.096488)
    (0.24,0.848287+ 0.094772)
    (0.26,0.762476+ 0.085141)
    (0.28,0.709274+ 0.068632)
    (0.30,0.763261+ 0.082038)
    (0.32,0.619350+ 0.097654)
    (0.34,0.648938+ 0.085098)
    (0.36,0.550380+ 0.081575)
    (0.38,0.572075+ 0.093867)
    (0.40,0.416135+ 0.095827)};
    \addplot[name path=NPR_bottom,color=mycolor1!70] coordinates {
    (0.1, 0.966636- 0.029638)
    (0.12,0.939154- 0.076717)
    (0.14,0.883484- 0.086820)
    (0.16,0.881140- 0.095176)
    (0.18,0.882741- 0.065913)
    (0.20,0.843096- 0.095891)
    (0.22,0.831645- 0.096488)
    (0.24,0.848287- 0.094772)
    (0.26,0.762476- 0.085141)
    (0.28,0.709274- 0.068632)
    (0.30,0.763261- 0.082038)
    (0.32,0.619350- 0.097654)
    (0.34,0.648938- 0.085098)
    (0.36,0.550380- 0.081575)
    (0.38,0.572075- 0.093867)
    (0.40,0.416135- 0.095827)};
    \addplot[mycolor1!50,fill opacity=0.5] 
    fill between[of=NPR_top and NPR_bottom];
    \addplot[name path=APPR_top,color=mycolor2!70] coordinates {
    (0.1, 0.8258+0.2876)
    (0.12,0.8409+0.0784)
    (0.14,0.8067+0.2066)
    (0.16,0.6989+0.2518)
    (0.18,0.6042+0.2717)
    (0.20,0.6087+0.2822)
    (0.22,0.5059+0.2576)
    (0.24,0.4957+0.2571)
    (0.26,0.3353+0.0845)
    (0.28,0.2560+0.0898)
    (0.30,0.2608+0.0991)
    (0.32,0.3532+0.0788)
    (0.34,0.2985+0.0488)
    (0.36,0.3094+0.0800)
    (0.38,0.2846+0.0972)
    (0.40,0.2393+0.0637)};
    \addplot[name path=APPR_bottom,color=mycolor2!70] coordinates {
    (0.1, 0.8258-0.2876)
    (0.12,0.8409-0.0784)
    (0.14,0.8067-0.2066)
    (0.16,0.6989-0.2518)
    (0.18,0.6042-0.2717)
    (0.20,0.6087-0.2822)
    (0.22,0.5059-0.2576)
    (0.24,0.4957-0.2571)
    (0.26,0.3353-0.0845)
    (0.28,0.2560-0.0898)
    (0.30,0.2608-0.0991)
    (0.32,0.3532-0.0788)
    (0.34,0.2985-0.0488)
    (0.36,0.3094-0.0800)
    (0.38,0.2846-0.0972)
    (0.40,0.2393-0.0637)};
    \addplot[mycolor2!50,fill opacity=0.2] 
    fill between[of=APPR_top and APPR_bottom];
    \addplot[name path=NPD_top,color=mycolor3!70] coordinates {
    (0.1, 0.9104+0.0688)
    (0.12,0.8500+0.0956)
    (0.14,0.8527+0.0837)
    (0.16,0.7571+0.1479)
    (0.18,0.7385+0.0973)
    (0.20,0.7096+0.2135)
    (0.22,0.6950+0.1682)
    (0.24,0.6840+0.2091)
    (0.26,0.4957+0.1562)
    (0.28,0.4563+0.1943)
    (0.30,0.5015+0.1857)
    (0.32,0.3468+0.0527)
    (0.34,0.3184+0.0381)
    (0.36,0.2934+0.0543)
    (0.38,0.2826+0.0734)
    (0.40,0.2401+0.0455)};
    \addplot[name path=NPD_bottom,color=mycolor3!70] coordinates {
    (0.1, 0.9104-0.0688)
    (0.12,0.8500-0.0956)
    (0.14,0.8527-0.0837)
    (0.16,0.7571-0.1479)
    (0.18,0.7385-0.0973)
    (0.20,0.7096-0.2135)
    (0.22,0.6950-0.1682)
    (0.24,0.6840-0.2091)
    (0.26,0.4957-0.1562)
    (0.28,0.4563-0.1943)
    (0.30,0.5015-0.1857)
    (0.32,0.3468-0.0527)
    (0.34,0.3184-0.0381)
    (0.36,0.2934-0.0543)
    (0.38,0.2826-0.0734)
    (0.40,0.2401-0.0455)};
    \addplot[mycolor3!50,fill opacity=0.2] 
    fill between[of=NPD_top and NPD_bottom];
    \addplot[name path=pLap_top,color=black!70] coordinates {
    (0.1, 0.8018+0.2290)
    (0.12,0.8149+0.1631)
    (0.14,0.7463+0.2400)
    (0.16,0.6420+0.2312)
    (0.18,0.5093+0.2631)
    (0.20,0.4791+0.3021)
    (0.22,0.3789+0.2404)
    (0.24,0.4825+0.2802)
    (0.26,0.2854+0.0714)
    (0.28,0.2204+0.0760)
    (0.30,0.2174+0.0931)
    (0.32,0.2953+0.0711)
    (0.34,0.2462+0.0449)
    (0.36,0.2538+0.0684)
    (0.38,0.2440+0.0887)
    (0.40,0.1965+0.0652)};
    \addplot[name path=pLap_bottom,color=black!70] coordinates {
    (0.1, 0.8018-0.2290)
    (0.12,0.8149-0.1631)
    (0.14,0.7463-0.2400)
    (0.16,0.6420-0.2312)
    (0.18,0.5093-0.2631)
    (0.20,0.4791-0.3021)
    (0.22,0.3789-0.2404)
    (0.24,0.4825-0.2802)
    (0.26,0.2854-0.0714)
    (0.28,0.2204-0.0760)
    (0.30,0.2174-0.0931)
    (0.32,0.2953-0.0711)
    (0.34,0.2462-0.0449)
    (0.36,0.2538-0.0684)
    (0.38,0.2440-0.0887)
    (0.40,0.1965-0.0652)};
    \addplot[black!50,fill opacity=0.2] 
    fill between[of=pLap_top and pLap_bottom];
  \end{axis}
  
\end{tikzpicture}}
	\caption{\label{fig:SBM_compare} (a) Conductance and the related (b) $Fscore$ of local clusters for an increasing number of inter-community edges in LFR graphs.}
\end{figure}

\subsection{Synthetic graphs}
\label{sec:synthgraphs}

\subsubsection{Communities benchmark}
\label{sec:lfr}

The Lancichinetti–Fortunato–Radicchi (LFR) model \cite{lfr2008community} produces  graphs with known communities. Vertex degree and community size distributions in the graphs follow power laws. A mixing parameter $\gamma$ is responsible for inter-community edges. The greater the value of $\gamma$, the more difficult it is to distinguish independent communities.

For the purpose of (non-weighted) graph generation, the number of vertices was 1,000 with an average vertex degree of 10 and a maximum of 50. Community sizes ranged from 20 to 100 vertices. The power law exponents for the vertex degree distribution and community size distribution were set to 2 and 1, respectively. Values of $\gamma$ were varied from 0.1 to 0.4.

It can be observed in Figure~\ref{fig:SBM_Fscore} that NPR leads to local clusters with the highest mean $Fscore$ for all values of $\gamma$, and the smallest overall mean standard deviation. Conductance-wise, NPR returns the best mean result for $\gamma < 0.28$, as can be seen in Figure~\ref{fig:SBM_CCut}. However, $p$-Diff takes the place of NPR when $\gamma \geq 0.28$. Again, the overall mean standard deviation is the smallest with NPR.

\begin{figure}[t!]
	\centering
	\begin{minipage}{\textwidth}
	\centering
	{\small \begin{tikzpicture}
  \centering
  \begin{customlegend}[
      legend columns=3,
      legend style={
      anchor=north,
      draw=none,
      /tikz/every even column/.append style={column sep=0.5cm}},
      legend entries={$\psi(x)$\\ $\lVert \nabla \psi(x) \rVert_\infty$\\ $\lambda$\\},
      legend image post style={xscale=1.0},]
    \addlegendimage{mycolor1, mark=*, mark options={fill=white}, thick}
    \addlegendimage{mycolor2, mark=square*, mark options={fill=white}, thick}
    \addlegendimage{mycolor3, mark=star, mark options={fill=white}, thick}
  \end{customlegend}
\end{tikzpicture}}
    \vspace{0.8em}
	\end{minipage}	
	\subcaptionbox{\label{fig:Converge_1}
    $p = 1.95$}
    {\small \begin{tikzpicture}[y=.2cm, x=.7cm]
  \centering
  \begin{axis}[
    width=0.32\columnwidth,
    grid=both,
    tick align=inside,
    yticklabel style={
        /pgf/number format/fixed,
        /pgf/number format/precision=5
},
scaled y ticks=false,
    xlabel style={align=center},
    xlabel={Iterations\\ \\$\Phi = 0.404$, $Fscore = 0.78$},
    ymode = log,
    legend pos=north west,
  ]
    \addplot[mycolor1, mark=*, mark options={fill=white}, thick] table[x=iter, y=Obj, header=true, col sep=comma] {ConvergeLFR_1.dat};
    \addplot[mycolor2, mark=square*, mark options={fill=white}, thick] table[x=iter, y=GradNorm, header=true, col sep=comma] {ConvergeLFR_1.dat};
    \addplot[mycolor3, mark=star, mark options={fill=white}, thick] table[x=iter, y=mu, header=true, col sep=comma] {ConvergeLFR_1.dat};

  \end{axis}
  
\end{tikzpicture}}%
	\subcaptionbox{\label{fig:Converge_2}
    $p = 1.8$}
    {\small \begin{tikzpicture}[y=.2cm, x=.7cm]
  \centering
  \begin{axis}[
    width=0.32\columnwidth,
    grid=both,
    tick align=inside,
    yticklabel style={
        /pgf/number format/fixed,
        /pgf/number format/precision=5
},
scaled y ticks=false,
    xlabel style={align=center},
    xlabel={Iterations\\ \\$\Phi = 0.387$, $Fscore = 0.80$},
    ymode = log,
    minor x tick num = 1,
    legend pos=north west,
  ]
    \addplot[mycolor1, mark=*, mark options={fill=white}, thick] table[x=iter, y=Obj, header=true, col sep=comma] {ConvergeLFR_2.dat};
    \addplot[mycolor2, mark=square*, mark options={fill=white}, thick] table[x=iter, y=GradNorm, header=true, col sep=comma] {ConvergeLFR_2.dat};
    \addplot[mycolor3, mark=star, mark options={fill=white}, thick] table[x=iter, y=mu, header=true, col sep=comma] {ConvergeLFR_2.dat};
  \end{axis}
  
\end{tikzpicture}}%
    \subcaptionbox{\label{fig:Converge_3}
    $p = 1.6$}
    {\small \begin{tikzpicture}[y=.2cm, x=.7cm]
  \centering
  \begin{axis}[
    width=0.32\columnwidth,
    grid=both,
    tick align=inside,
    yticklabel style={
        /pgf/number format/fixed,
        /pgf/number format/precision=5
},
scaled y ticks=false,
    xlabel style={align=center},
    xlabel={Iterations\\ \\$\Phi = 0.359$, $Fscore = 0.83$},
    ymode = log,
    minor x tick num = 1,
    legend pos=north west,
  ]
    \addplot[mycolor1, mark=*, mark options={fill=white}, thick] table[x=iter, y=Obj, header=true, col sep=comma] {ConvergeLFR_3.dat};
    \addplot[mycolor2, mark=square*, mark options={fill=white}, thick] table[x=iter, y=GradNorm, header=true, col sep=comma] {ConvergeLFR_3.dat};
    \addplot[mycolor3, mark=star, mark options={fill=white}, thick] table[x=iter, y=mu, header=true, col sep=comma] {ConvergeLFR_3.dat}; 
  \end{axis}
  
\end{tikzpicture}}
	\caption{\label{fig:Converge} Levenberg-Marquardt method for the NPR problem defined on an LFR graph with $\gamma = 0.3$ at (a) $p = 1.95$, (b) $p = 1.8$ and (c) $p = 1.6$.}
\end{figure}
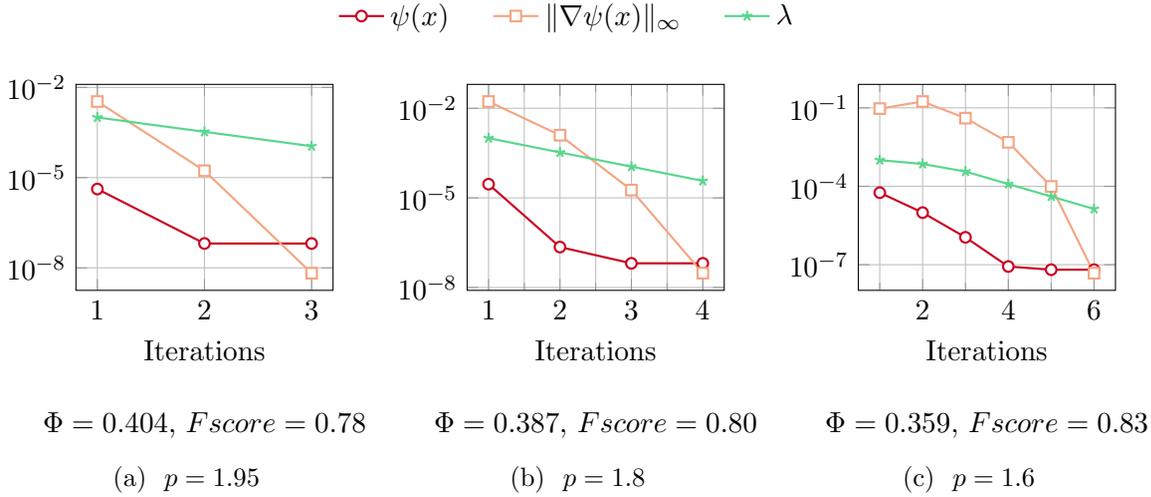

Plots in Figure~\ref{fig:Converge} illustrate the convergence history of the Levenberg-Marquardt method for solution of the NPR problem defined on a graph with $\gamma = 0.3$ at $p = 1.95, 1.8$ and $1.6$. It can be seen that the number of iterations to achieve convergence increases slightly as $p$ goes to $1.6$. Note the decrease in mean conductance from $0.404$ to $0.359$ and the increase in mean $Fscore$ from $0.78$ to $0.83$.

\begin{figure}[t!]
	\centering
    \begin{minipage}{0.99\textwidth}
    \subcaptionbox{\label{fig:Gauss8_nodes}}
    {\includegraphics[width=0.48\columnwidth]{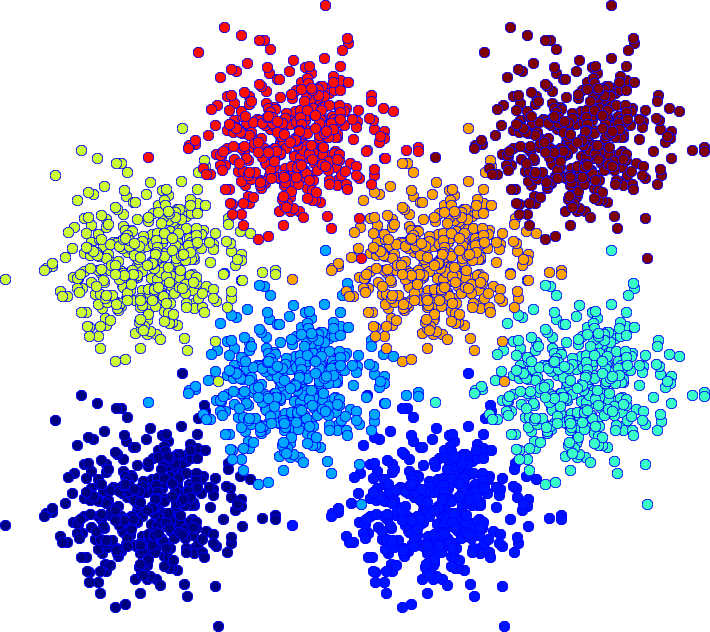}}    
    \subcaptionbox{\label{fig:Gauss_Clusters_re}}
    {
    \includegraphics[width=0.48\columnwidth]{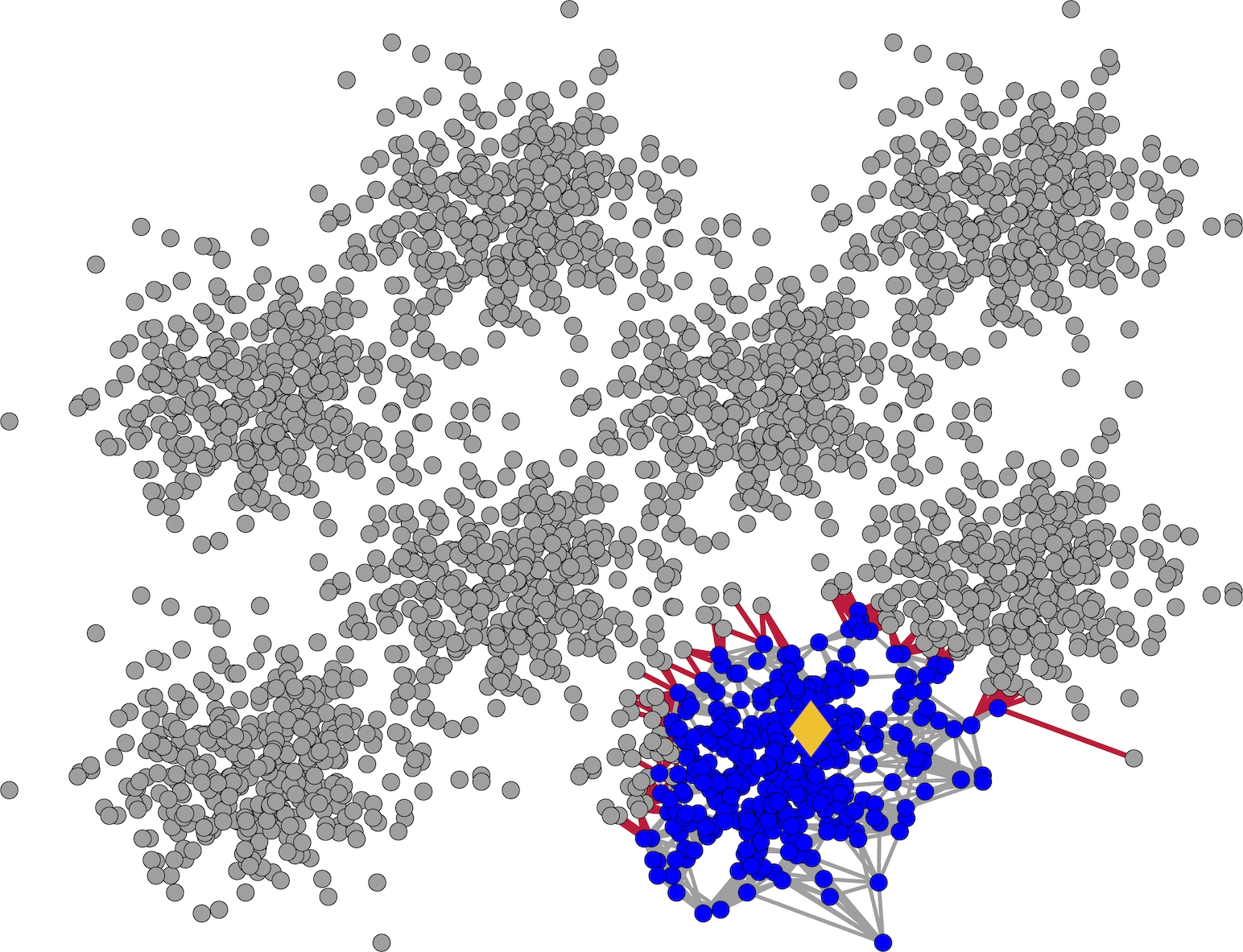}}
    \end{minipage}\\
	\begin{minipage}{0.99\textwidth}
	\centering
	\vspace{1.5em}
	\begin{tikzpicture}
  \centering
  \begin{customlegend}[
      legend columns=4,
      legend style={
      anchor=north,
      draw=none,
      /tikz/every even column/.append style={column sep=0.5cm}},
      legend entries={NPR\\ APPR\\ NPD\\ $p$-DIFF\\},
      legend image post style={xscale=1.0},]
    \addlegendimage{mycolor1, mark=*, mark options={fill=white}, thick}
    \addlegendimage{mycolor2, mark=square*, mark options={fill=white}, thick}
    \addlegendimage{mycolor3, mark=star, mark options={fill=white}, thick}
    \addlegendimage{black, mark=triangle*, mark options={fill=white}, thick}
    \addlegendimage{mycolor5, mark=diamond*, mark options={fill=white}, thick}
  \end{customlegend}
\end{tikzpicture}
    \vspace{0.8em}
	\subcaptionbox{\label{fig:Gauss_CCut}}
    {\begin{tikzpicture}[y=.2cm, x=.7cm]
  \centering
  \begin{axis}[
    width=0.49\columnwidth,
    grid=both,
    tick align=inside,
    yticklabel style={
        /pgf/number format/fixed,
        /pgf/number format/precision=5
},
scaled y ticks=false,
    xlabel={Number of groupings},
    ylabel={Conductance ($\Phi$)},
    minor x tick num = 1,
    legend pos=north west,
  ]
    \addplot[mycolor1, mark=*, mark options={fill=white}, thick] table[x=k, y=NPR, header=true, col sep=comma] {Gauss_CCUt.dat};
    \addplot[mycolor2, mark=square*, mark options={fill=white}, thick] table[x=k, y=APPR, header=true, col sep=comma] {Gauss_CCUt.dat};
    \addplot[mycolor3, mark=star, mark options={fill=white}, thick] table[x=k, y=NPD, header=true, col sep=comma] {Gauss_CCUt.dat};
    \addplot[black, mark=triangle*, mark options={fill=white}, thick] table[x=k, y=p-Lap, header=true, col sep=comma] {Gauss_CCUt.dat};        
    \addplot[name path=NPR_top,color=mycolor1!70] coordinates {
    ( 2,0.000959 + 0.004226 )
    ( 5,0.006417 + 0.003991 )
    ( 8,0.010030 + 0.002963 )
    (13,0.011574 + 0.004233 )
    (18,0.016600 + 0.002968 )
    (25,0.016309 + 0.003226 )
    (32,0.016446 + 0.004546 )
    (41,0.017456 + 0.004039 )
    };
    \addplot[name path=NPR_bottom,color=mycolor1!70] coordinates {
    ( 2,0.000959 - 0.004226 )
    ( 5,0.006417 - 0.003991 )
    ( 8,0.010030 - 0.002963 )
    (13,0.011574 - 0.004233 )
    (18,0.016600 - 0.002968 )
    (25,0.016309 - 0.003226 )
    (32,0.016446 - 0.004546 )
    (41,0.017456 - 0.004039 )
    };
    \addplot[mycolor1!50,fill opacity=0.5] 
    fill between[of=NPR_top and NPR_bottom];
    \addplot[name path=APPR_top,color=mycolor1!70] coordinates {
    ( 2,0.011011 + 0.0130)
    ( 5,0.012972 + 0.0050)
    ( 8,0.014601 + 0.0023)
    (13,0.015678 + 0.0013)
    (18,0.015625 + 0.0017)
    (25,0.014217 + 0.0010)
    (32,0.013644 + 0.0009)
    (41,0.012619 + 0.0008)
    };
    \addplot[name path=APPR_bottom,color=mycolor1!70] coordinates {
    ( 2,0.011011 - 0.0130)
    ( 5,0.012972 - 0.0050)
    ( 8,0.014601 - 0.0023)
    (13,0.015678 - 0.0013)
    (18,0.015625 - 0.0017)
    (25,0.014217 - 0.0010)
    (32,0.013644 - 0.0009)
    (41,0.012619 - 0.0008)
    };
    \addplot[mycolor2!50,fill opacity=0.5] 
    fill between[of=APPR_top and APPR_bottom];
    \addplot[name path=NPD_top,color=mycolor3!70] coordinates {
    ( 2,0.009630 + 0.010830)
    ( 5,0.021912 + 0.011947)
    ( 8,0.022625 + 0.008671)
    (13,0.026449 + 0.009127)
    (18,0.028618 + 0.007737)
    (25,0.030101 + 0.008079)
    (32,0.031249 + 0.007914)
    (41,0.031383 + 0.009373)
    };
    \addplot[name path=NPD_bottom,color=mycolor3!70] coordinates {
    ( 2,0.009630 - 0.010830)
    ( 5,0.021912 - 0.011947)
    ( 8,0.022625 - 0.008671)
    (13,0.026449 - 0.009127)
    (18,0.028618 - 0.007737)
    (25,0.030101 - 0.008079)
    (32,0.031249 - 0.007914)
    (41,0.031383 - 0.009373)
    };
    \addplot[mycolor3!50,fill opacity=0.2] 
    fill between[of=NPD_top and NPD_bottom];
    \addplot[name path=pLap_top,color=black!70] coordinates {
    ( 2,0.018595 + 0.0131)
    ( 5,0.019981 + 0.0060)
    ( 8,0.014903 + 0.0064)
    (13,0.010710 + 0.0042)
    (18,0.009268 + 0.0030)
    (25,0.008529 + 0.0033)
    (32,0.007615 + 0.0019)
    (41,0.008025 + 0.0022)
    };
    \addplot[name path=pLap_bottom,color=black!70] coordinates {
    ( 2,0.018595 - 0.0131)
    ( 5,0.019981 - 0.0060)
    ( 8,0.014903 - 0.0064)
    (13,0.010710 - 0.0042)
    (18,0.009268 - 0.0030)
    (25,0.008529 - 0.0033)
    (32,0.007615 - 0.0019)
    (41,0.008025 - 0.0022)
    };
    \addplot[black!50,fill opacity=0.2] 
    fill between[of=pLap_top and pLap_bottom];
  \end{axis}
  
\end{tikzpicture}
    }\hfill%
	\subcaptionbox{\label{fig:Gauss_Fscore}}
    {\begin{tikzpicture}[y=.2cm, x=.7cm]
  \centering
  \begin{axis}[
    width=0.49\columnwidth,
    grid=both,
    tick align=inside,
    yticklabel style={
        /pgf/number format/fixed,
        /pgf/number format/precision=5
},
scaled y ticks=false,
    xlabel={Number of groupings},
    ylabel={$Fscore$ (corresponding to $\Phi$)},
    minor x tick num = 1,
    ymax=1,
    ymin=0.0,
    legend pos=north west,
  ]
    \addplot[mycolor1, mark=*, mark options={fill=white}, thick] table[x=k, y=NPR, header=true, col sep=comma] {Gauss_Fscore.dat};
    \addplot[mycolor2, mark=square*, mark options={fill=white}, thick] table[x=k, y=APPR, header=true, col sep=comma] {Gauss_Fscore.dat};
    \addplot[mycolor3, mark=star, mark options={fill=white}, thick] table[x=k, y=NPD, header=true, col sep=comma] {Gauss_Fscore.dat};
    \addplot[black, mark=triangle*, mark options={fill=white}, thick] table[x=k, y=p-Lap, header=true, col sep=comma] {Gauss_Fscore.dat};        
    \addplot[name path=NPR_top,color=mycolor1!70] coordinates {
    ( 2,0.998749 + 0.000002)
    ( 5,0.820724 + 0.066047)
    ( 8,0.805514 + 0.071379)
    (13,0.736567 + 0.089198)
    (18,0.7773   + 0.078041)
    (25,0.632138 + 0.095737)
    (32,0.649556 + 0.102026)
    (41,0.662327 + 0.093831)
    };
    \addplot[name path=NPR_bottom,color=mycolor1!70] coordinates {
    ( 2,0.998749 - 0.000002)
    ( 5,0.820724 - 0.066047)
    ( 8,0.805514 - 0.071379)
    (13,0.736567 - 0.089198)
    (18,0.7773   - 0.078041)
    (25,0.632138 - 0.095737)
    (32,0.649556 - 0.102026)
    (41,0.662327 - 0.093831)
    };
    \addplot[mycolor1!50,fill opacity=0.5] 
    fill between[of=NPR_top and NPR_bottom];
    \addplot[name path=APPR_top,color=mycolor1!70] coordinates {
    ( 2,0.749443 + 0.2030)
    ( 5,0.500524 + 0.1729)
    ( 8,0.363333 + 0.0632)
    (13,0.150252 + 0.1071)
    (18,0.117961 + 0.0827)
    (25,0.109219 + 0.0592)
    (32,0.058915 + 0.0506)
    (41,0.064025 + 0.0617)
    };
    \addplot[name path=APPR_bottom,color=mycolor1!70] coordinates {
    ( 2,0.749443 - 0.2030)
    ( 5,0.500524 - 0.1729)
    ( 8,0.363333 - 0.0632)
    (13,0.150252 - 0.1071)
    (18,0.117961 - 0.0827)
    (25,0.109219 - 0.0592)
    (32,0.058915 - 0.0506)
    (41,0.064025 - 0.0617)
    };
    \addplot[mycolor2!50,fill opacity=0.5] 
    fill between[of=APPR_top and APPR_bottom];
    \addplot[name path=NPD_top,color=mycolor3!70] coordinates {
    ( 2,0.829686 +     0.103977)
    ( 5,0.717904 +     0.100358)
    ( 8,0.639630 +     0.102785)
    (13,0.623668 +     0.101937)
    (18,0.615570 +     0.108101)
    (25,0.520826 +     0.106191)
    (32,0.557174 +     0.107110)
    (41,0.597433 +     0.107688)
    };
    \addplot[name path=NPD_bottom,color=mycolor3!70] coordinates {
    ( 2,0.829686 -     0.103977)
    ( 5,0.717904 -     0.100358)
    ( 8,0.639630 -     0.102785)
    (13,0.623668 -     0.101937)
    (18,0.615570 -     0.108101)
    (25,0.520826 -     0.106191)
    (32,0.557174 -     0.107110)
    (41,0.597433 -     0.107688)
    };
    \addplot[mycolor3!50,fill opacity=0.2] 
    fill between[of=NPD_top and NPD_bottom];
    \addplot[name path=pLap_top,color=black!70] coordinates {
    ( 2,0.837654 +    0.1468)
    ( 5,0.651273 +    0.1592)
    ( 8,0.394285 +    0.0302)
    (13,0.276920 +    0.0464)
    (18,0.195544 +    0.0201)
    (25,0.158825 +    0.0177)
    (32,0.166441 +    0.0335)
    (41,0.143068 +    0.0217)
    };
    \addplot[name path=pLap_bottom,color=black!70] coordinates {
    ( 2,0.837654 -    0.1468)
    ( 5,0.651273 -    0.1592)
    ( 8,0.394285 -    0.0302)
    (13,0.276920 -    0.0464)
    (18,0.195544 -    0.0201)
    (25,0.158825 -    0.0177)
    (32,0.166441 -    0.0335)
    (41,0.143068 -    0.0217)
    };
    \addplot[black!50,fill opacity=0.2] 
    fill between[of=pLap_top and pLap_bottom];
  \end{axis}
  
\end{tikzpicture}}
    \end{minipage}
	\caption{\label{fig:Gauss_compare} (a) Example of eight groupings. (b) Local cluster based on the NPR problem solution. The starting vertex is shown as a yellow diamond and the edges comprising the boundary are in red. (c) Conductance and the related (d) $Fscore$ of local clusters for an increasing number of (400 point) groupings.}
\end{figure}

\subsubsection{Gaussian datasets}
\label{sec:gauss}

Groups of 400 points in $\mathbb{R}^2$ randomly sampled (in each dimension) from a normal distribution with a variance of 0.055 were generated and their centres located equidistantly on a square grid. Connectivity between points was established by a nearest neighbour search set to 10 points, i.e., the 10 closest (distance-wise) points to every point were considered connected to that point.

Graphs representing 2, 5, 8, 13, 18, 25, 32 and 41 groupings were constructed. While points are the geometric manifestation or realisation of vertices, edges correspond to the connections between points. The number of edges in each graph was 4,902, 12,153, 19,319, 31,204, 43,013, 59,644, 76,044 and 97,112. Let points be denoted by $y$. For every graph $\mathcal{G} = (\mathcal{V}, \mathcal{E}, w)$ and all $\{u, v \} \in \mathcal{E}$, the assigned weight was given by
\begin{equation*}
	w (u, v) = \exp \left( -4 \frac{\lVert y(u) - y(v) \rVert_2^2}{\nu^2} \right) ,
\end{equation*}
where $\nu$ is the greater distance between $y(u)$ or $y(v)$ and their respective tenth nearest neighbour. 

Figure~\ref{fig:Gauss8_nodes} depicts an instance of eight groupings as identified by their colour. A local cluster based on the NPR problem solution is highlighted along with the relevant random starting vertex and edge boundary in Figure~\ref{fig:Gauss_Clusters_re}. The NPR problem was solved with $\beta = 0.0001$ for the 2 and 5 groupings, $\beta = 0.001$ for the 8 and 13 groupings, and $\beta = 0.005$ for the 18 and more groupings. It can be clearly seen from Figure~\ref{fig:Gauss_Fscore} that NPR has the highest mean $Fscore$ for any number of groupings. In comparison, the accuracy of both APPR and $p$-DIFF deteriorates significantly as the number of groupings increases. NPD does not perform badly, but this is at the expense of a greater number of iterations (ten times that suggested in~\cite{ibrahim2019nonlinear}) involved in the solution methodology. With regard to conductance, $p$-DIFF returns the lowest mean values for more than eight groupings, which is visible in Figure~\ref{fig:Gauss_CCut}.

\begin{figure}[t!]
	\centering
	\begin{minipage}{\textwidth}
	\centering
	{\small \begin{tikzpicture}
  \centering
  \begin{customlegend}[
      legend columns=3,
      legend style={
      anchor=north,
      draw=none,
      /tikz/every even column/.append style={column sep=0.5cm}},
      legend entries={$\psi(x)$\\ $\lVert \nabla \psi(x) \rVert_\infty$\\ $\lambda$\\},
      legend image post style={xscale=1.0},]
    \addlegendimage{mycolor1, mark=*, mark options={fill=white}, thick}
    \addlegendimage{mycolor2, mark=square*, mark options={fill=white}, thick}
    \addlegendimage{mycolor3, mark=star, mark options={fill=white}, thick}
  \end{customlegend}
\end{tikzpicture}}
    \vspace{0.8em}
	\end{minipage}
	\subcaptionbox{\label{fig:ConvergeGauss_1}
    $p = 1.95$}
    {\small \begin{tikzpicture}[y=.2cm, x=.7cm]
  \centering
  \begin{axis}[
    width=0.32\columnwidth,
    grid=both,
    tick align=inside,
    yticklabel style={
        /pgf/number format/fixed,
        /pgf/number format/precision=5
},
scaled y ticks=false,
    xlabel style={align=center},
    xlabel={Iterations\\ \\$\Phi = 0.014$, $Fscore = 0.82$},
    ymode = log,
    legend pos=north west,
    minor x tick num = 2,
  ]
    \addplot[mycolor1, mark=*, mark options={fill=white}, thick] table[x=iter, y=Obj, header=true, col sep=comma] {ConvergeGauss8_1.dat};
    \addplot[mycolor2, mark=square*, mark options={fill=white}, thick] table[x=iter, y=GradNorm, header=true, col sep=comma] {ConvergeGauss8_1.dat};
    \addplot[mycolor3, mark=star, mark options={fill=white}, thick] table[x=iter, y=mu, header=true, col sep=comma] {ConvergeGauss8_1.dat};
  \end{axis}
  
\end{tikzpicture}}%
	\subcaptionbox{\label{fig:ConvergeGauss_2}
    $p = 1.8$}
    {\small \begin{tikzpicture}[y=.2cm, x=.7cm]
  \centering
  \begin{axis}[
    width=0.32\columnwidth,
    grid=both,
    tick align=inside,
    yticklabel style={
        /pgf/number format/fixed,
        /pgf/number format/precision=5
},
scaled y ticks=false,
    xlabel style={align=center},
    xlabel={Iterations\\ \\$\Phi = 0.013$, $Fscore = 0.85$},
    ymode = log,
    minor x tick num = 1,
    legend pos=north west,
  ]
    \addplot[mycolor1, mark=*, mark options={fill=white}, thick] table[x=iter, y=Obj, header=true, col sep=comma] {ConvergeGauss8_2.dat};
    \addplot[mycolor2, mark=square*, mark options={fill=white}, thick] table[x=iter, y=GradNorm, header=true, col sep=comma] {ConvergeGauss8_2.dat};
    \addplot[mycolor3, mark=star, mark options={fill=white}, thick] table[x=iter, y=mu, header=true, col sep=comma] {ConvergeGauss8_2.dat};
  \end{axis}
  
\end{tikzpicture}}%
    \subcaptionbox{\label{fig:ConvergeGauss_3}
    $p = 1.6$}
    {\small \begin{tikzpicture}[y=.2cm, x=.7cm]
  \centering
  \begin{axis}[
    width=0.32\columnwidth,
    grid=both,
    tick align=inside,
    yticklabel style={
        /pgf/number format/fixed,
        /pgf/number format/precision=5
},
scaled y ticks=false,
    xlabel style={align=center},
    xlabel={Iterations\\ \\$\Phi = 0.012$, $Fscore = 0.87$},
    ymode = log,
    minor x tick num = 1,
    legend pos=north west,
  ]
    \addplot[mycolor1, mark=*, mark options={fill=white}, thick] table[x=iter, y=Obj, header=true, col sep=comma] {ConvergeGauss8_3.dat};
    \addplot[mycolor2, mark=square*, mark options={fill=white}, thick] table[x=iter, y=GradNorm, header=true, col sep=comma] {ConvergeGauss8_3.dat};
    \addplot[mycolor3, mark=star, mark options={fill=white}, thick] table[x=iter, y=mu, header=true, col sep=comma] {ConvergeGauss8_3.dat}; 
  \end{axis}
  
\end{tikzpicture}}
	\caption{\label{fig:ConvergeGauss} Levenberg-Marquardt method for the NPR problem defined on a graph with eight groupings at (a) $p = 1.95$, (b) $p = 1.8$ and (c) $p = 1.6$.}
\end{figure}
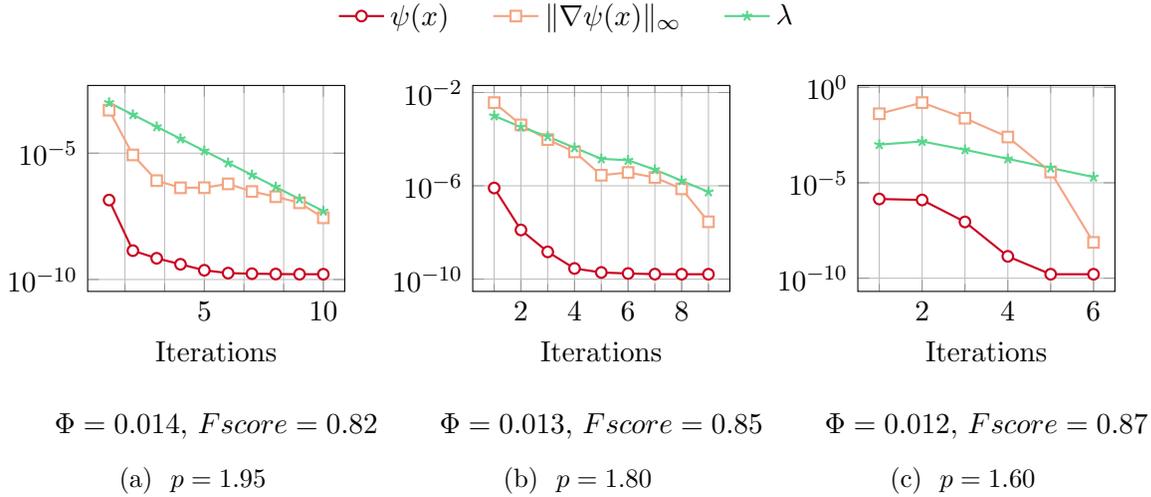

Plots in Figure~\ref{fig:ConvergeGauss} demonstrate the performance of the Levenberg-Marquardt method for solution of the NPR problem defined on a graph accounting for eight groupings at $p = 1.95, 1.8$ and $1.6$. As opposed to before, the number of iterations to achieve convergence actually decreases a little as $p$ goes to $1.6$. Note the decrease in mean conductance from $0.014$ to $0.012$ and the increase in mean $Fscore$ from $0.82$ to $0.87$.

\subsection{Real-word inspired graphs}
\label{sec:realgraphs}

\subsubsection{Image classification}
\label{sec:image}

There is quite the number of image databases that have been developed in order to facilitate research in pattern recognition and machine learning. The MNIST database~\cite{deng2012mnist} of handwritten digits from zero to nine has been used extensively. A test set of 10,000 $28 \times 28$ grayscale images is provided as part of the database. Similarly, the USPS database~\cite{hull1994USPS} contains 11,000 handwritten digit $16 \times 16$ grayscale images that were scanned from envelopes in a working post office (by the U.S. Postal Service). Fashion-MNIST~\cite{xiao2017fashion} offers a test set of 10,000 $28 \times 28$ grayscale images of clothing items from (online retailer) Zalando's website. Each image is categorised into one of ten classes:  t-shirt/top, trouser, pullover, dress, coat, sandals, shirt, sneaker, bag and ankle boots.

\begin{table}[ht!]
\setlength{\tabcolsep}{1.5pt}
    \centering
    \caption{Mean conductance and the related mean $Fscore$ of local clusters representing subsets of images in the following databases: MNIST, Fashion-MNIST and USPS.}
    \label{tab:Images_Classification}
    \resizebox{\columnwidth}{!}{%
	\begin{tabular}{l|ccc|ccc|ccc}
	    \hline
	        & &  MNIST & & & Fashion-MNIST & & & USPS & \\
	        & $\Phi$ & & $Fscore$ & $\Phi$ & & $Fscore$ & $\Phi$ & & $Fscore$ \\
	    \hline
	    \hline
	    NPR    & $\mathbf{0.061}\scriptstyle{\pm 1\cdot 10^{-2}}$  & &   $\mathbf{0.710}\scriptstyle{\pm2\cdot 10^{-1}}$                  
	           & $\mathbf{0.026}\scriptstyle{\pm 1\cdot 10^{-2}}$  & &   $\mathbf{0.586}\scriptstyle{\pm2\cdot 10^{-1}}$                  
	           & $0.094\scriptstyle{\pm 3\cdot 10^{-2}}$           & &   $\mathbf{0.519}\scriptstyle{\pm1\cdot 10^{-1}}$   \\
	    APPR   & $0.282\scriptstyle{\pm 5\cdot 10^{-2}}$           & &   $0.317\scriptstyle{\pm1\cdot 10^{-1}}$                  
	           & $0.234\scriptstyle{\pm 4\cdot 10^{-2}}$           & &   $0.279\scriptstyle{\pm2\cdot 10^{-1}}$                  
	           & $\mathbf{0.071}\scriptstyle{\pm 1\cdot 10^{-2}}$  & &   $0.211\scriptstyle{\pm2\cdot 10^{-2}}$            \\
	    NPD    & $0.135\scriptstyle{\pm 1\cdot 10^{-2}}$           & &   $0.550\scriptstyle{\pm2\cdot 10^{-1}}$                                   
	           & $0.119\scriptstyle{\pm 4\cdot 10^{-2}}$           & &   $0.451\scriptstyle{\pm2\cdot 10^{-1}}$                                   
	           & $0.144\scriptstyle{\pm 2\cdot 10^{-2}}$           & &   $0.403\scriptstyle{\pm2\cdot 10^{-1}}$            \\
	$p$-DIFF   & $0.111\scriptstyle{\pm 3\cdot 10^{-2}}$           & &   $0.336\scriptstyle{\pm1\cdot 10^{-1}}$                                   
	           & $0.043\scriptstyle{\pm 1\cdot 10^{-2}}$           & &   $0.317\scriptstyle{\pm9\cdot 10^{-2}}$                                   
	           & $0.116\scriptstyle{\pm 2\cdot 10^{-2}}$           & &   $0.312\scriptstyle{\pm2\cdot 10^{-2}}$            \\
	    \hline
	\end{tabular}
	}
\end{table}

Images can be treated as points in Euclidean space of dimension equal to their resolution (i.e., the total number of pixels). As such, graphs were constructed for the aforementioned databases as was done previously. Table~\ref{tab:Images_Classification} presents the mean conductance and related mean $Fscore$ associated with local clusters. Those based on the NPR problem solution have the highest mean $Fscore$ and the lowest conductance in all but one case, which is the USPS database.

\begin{figure}[t!]
    \centering
    \caption*{\large{\textbf{Distance (km)}}}
    \begin{subfigure}[b]{0.33\textwidth}
      \includegraphics[width=\textwidth]{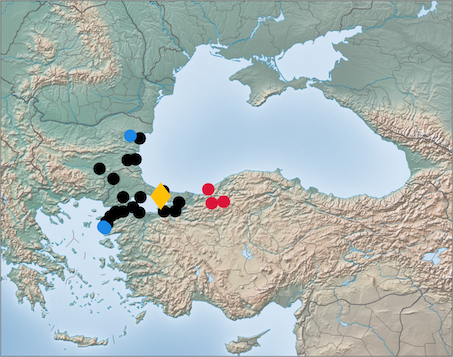}
      \caption{$Fscore = 0.875$}
    \end{subfigure}\hfill
    \begin{subfigure}[b]{0.33\textwidth}
      \includegraphics[width=\textwidth]{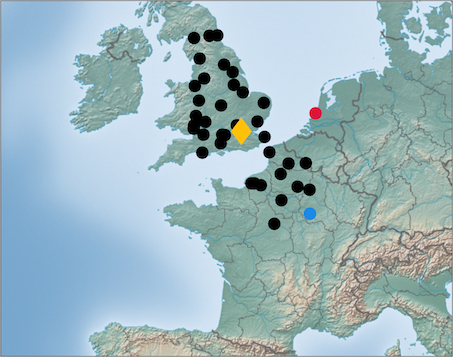}
      \caption{$Fscore = 0.974$}
    \end{subfigure}\hfill
    \begin{subfigure}[b]{0.33\textwidth}
      \includegraphics[width=\textwidth]{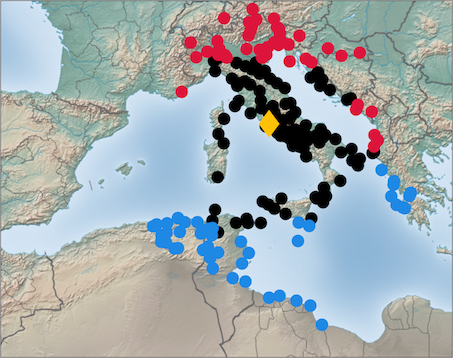}
      \caption{$Fscore = 0.709$}
    \end{subfigure}
    \vfill
    \caption*{\large{\textbf{Duration (days)}}}
    \begin{subfigure}[b]{0.33\textwidth}
      \includegraphics[width=\textwidth]{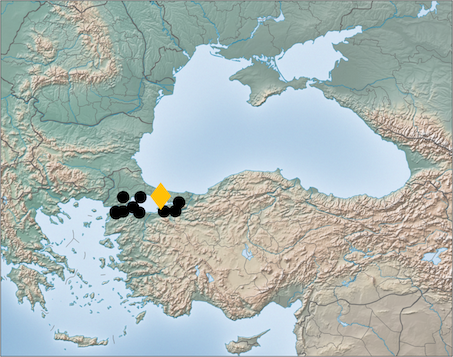}
      \caption{$Fscore = 1$}
    \end{subfigure}\hfill
    \begin{subfigure}[b]{0.33\textwidth}
      \includegraphics[width=\textwidth]{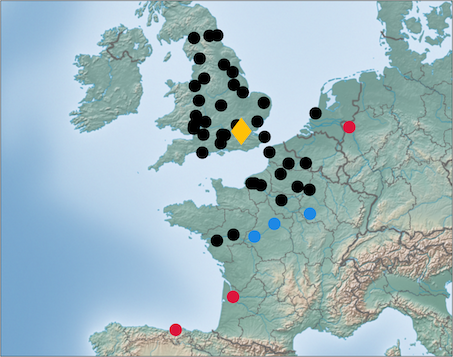}
      \caption{$Fscore = 0.929$}
    \end{subfigure}\hfill
    \begin{subfigure}[b]{0.33\textwidth}
      \includegraphics[width=\textwidth]{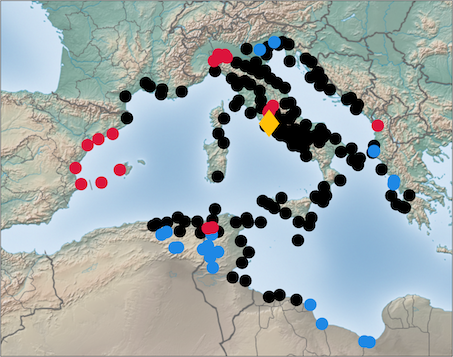}
      \caption{$Fscore = 0.885$}
    \end{subfigure}
    \vfill
    \caption*{\large{\textbf{Financial cost (denarii)}}}
    \begin{subfigure}[b]{0.33\textwidth}
      \includegraphics[width=\textwidth]{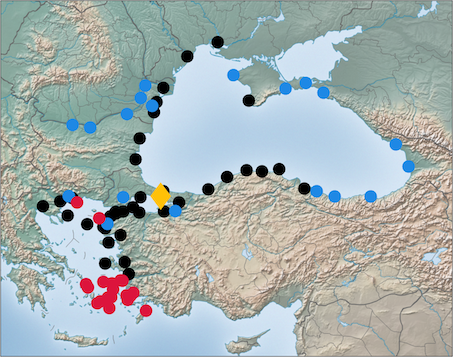}
      \caption{$Fscore = 0.672$}
    \end{subfigure}\hfill
    \begin{subfigure}[b]{0.33\textwidth}
      \includegraphics[width=\textwidth]{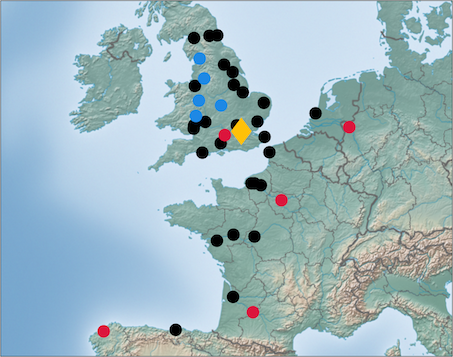}
      \caption{$Fscore = 0.848$}
    \end{subfigure}\hfill
    \begin{subfigure}[b]{0.33\textwidth}
      \includegraphics[width=\textwidth]{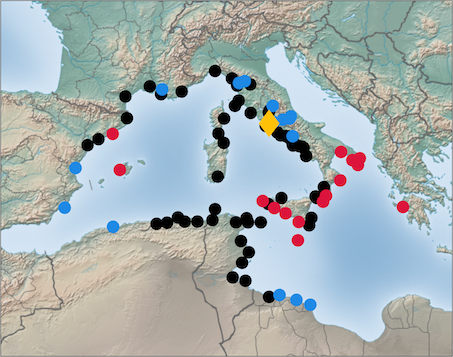}
      \caption{$Fscore = 0.803$}
    \end{subfigure}
    \caption{Local clusters based on the NPR problem solution in black and blue, and their ground truth equivalent with the same number of vertices in black and red. The number of vertices belonging to a local cluster is (a) 24, (b) 38, (c) 151, (d) 12, (e) 42, (f) 165, (g) 58, (h) 33 and (i) 76. Yellow diamonds represent Constantinopolis (a,d,g), Londinium (b,e,h) and Roma (c,f,i). Clusters are related to the distance (a,b,c), duration (d,e,f) and financial cost (g,h,i) between sites.}
    \label{fig:Orbis_clusters}
\end{figure}

\subsubsection{Roman world transport network circa 200 AD}
\label{sec:rome}

ORBIS: The Stanford Geospatial Network Model of the Roman World~\cite{scheidel2015orbis} provides the distance, time and financial cost associated with travel circa 200 AD. The model encompasses a large number of settlements, roads, navigable rivers, and sea lanes that framed movement across the Roman Empire at the time. A disclaimer is in order at this point. There is no intention here whatsoever to make any historically relevant claims. It is only desired to demonstrate the generality of application.

Three graphs were constructed, each with 677 vertices and 1104 edges representing sites and travel routes, respectively. They differed only in weight function definition. For every graph $\mathcal{G} = (\mathcal{V}, \mathcal{E}, w)$ and all $\{u, v \} \in \mathcal{E}$, the assigned weight was given by
\begin{equation*}
	w (u, v) = \exp \left( -2 \frac{\chi^2}{\iota^2} \right) ,
\end{equation*}
where $\chi (u, v)$ is the distance, duration, or financial cost between the sites represented by vertices $u$ and $v$. The raw data from ORBIS can have a value for one direction between vertices $u$ and $v$, and another for the opposition direction. An average was, therefore, taken between two values in all cases. $\iota$ is the mean value of the distance, duration, or financial cost between $u$ and adjacent vertices.

A NPR problem solution was obtained for every graph and starting vertex, of which there was three, representing the cities of Constantinopolis (modern day Istanbul, Turkey), Londinium (modern day London, England), and Roma (modern day Rome, Italy). Local clusters were then formed. In order to have a baseline, Dijkstra’s algorithm was utilised to discover the shortest paths (be it distance, time, or financial cost related) from a starting vertex. This provided the rationale for supposed ground truth local clusters. The number of vertices in these clusters was taken to be the same as those in the local clusters based on the NPR problem solution. 

Figure~\ref{fig:Orbis_clusters} visualises the local clusters and their supposed ground truth counterparts. Starting vertices are illustrated by yellow diamonds. Sites coloured black are those that are part of both clusters. While blue sites are those that were not part of the ground truth cluster, the red sites are those that were only part of the ground truth cluster. Other than for two cases, the clusters based on the NPR problem solution were in very good to excellent agreement with the ground truth clusters. Even for those two cases, performance was still respectable with an $Fscore = 0.709$ and $Fscore = 0.672$. An $Fscore = 1$ was registered for the graph with the travel duration edge weight and the starting vertex of Constantinopolis.

\clearpage

\section*{Acknowledgments}
The authors gratefully acknowledge the scientific support and HPC resources provided by the Erlangen National High Performance Computing Center (NHR@FAU) of the Friedrich-Alexander-Universität Erlangen-Nürnberg (FAU) under the NHR project 80227. NHR funding is provided by federal and Bavarian state authorities. NHR@FAU hardware is partially funded by the German Research Foundation (DFG) -- 440719683. A great debt of gratitude is owed to Prof.~Horia Cornean for his review of the manuscript. His advice and support significantly enhanced the quality of the article.

\bibliographystyle{elsarticle-num}
\bibliography{references}

\end{document}